\documentclass[10pt]{amsart}

\usepackage{amsmath, amsthm, amssymb, graphicx, enumerate, color, mathrsfs, changepage, accents, mathtools}

\newcounter{citedtheorems}

\newtheorem{defn}{Definition}[section]
\newtheorem{defns}[defn]{Definitions}
\newtheorem{theorem}[defn]{Theorem}
\newtheorem{theorem-m1}[defn]{Main Theorem}
\newtheorem{theorem-c1}[defn]{Central Theorem}

\newtheorem{ex}[defn]{Exercise}
\newcommand{\ord}{\operatorname{Or}}
\newcommand{\tr}{\operatorname{Tr}}
\newcommand{\cs}{\mathbf{s}}
\newcommand{\ma}{\mathbf{a}}
\newcommand{\vp}{\varphi} 
\newcommand{\mct}{\mathcal{T}}
\newcommand{\tlf}{\trianglelefteq}
\newcommand{\ts}{\mathbf{S}}
\newcommand{\br}{\vspace{3mm}}
\newcommand{\sbr}{\vspace{1mm}}
\newcommand{\prt}{\operatorname{Par}}

\newcommand{\de}{\mathcal{D}}
\newcommand{\lost}{\L os' }

\newcommand{\mc}{\mathcal{C}}
\newcommand{\mb}{\mathbf{b}}
\newcommand{\ml}{\mathcal{L}}
\newcommand{\lgn}{\operatorname{lgn}}
\newcommand{\maxdom}{\operatorname{maxdom}}
\newcommand{\rstr}{\upharpoonright}
\newcommand{\vv}{\mathbf{V}}

\newtheorem{theorem-kl}[defn]{Key Lemma}

\newtheorem*{theorem-x}{Theorem}
\newtheorem*{theorem-m}{Main Theorem}
\newtheorem*{theorem-n}{Main Theorem}
\newtheorem*{cor-x}{Corollary}
\newtheorem*{lemma-x}{Lemma}
\newtheorem*{concl-x}{Conclusion}
\newtheorem*{claim-x}{Claim}
\newtheorem*{thm-r}{Theorem \ref{concl:sop2-max}}
\newtheorem*{thm-q}{Theorem \ref{theorem:p-t}}
\newtheorem*{claim-s}{Claim \ref{m1-sat}}

\newtheorem{mqst}[defn]{Central question}

\newtheorem{thm-lit}[citedtheorems]{Theorem}
\newtheorem{defn-lit}[citedtheorems]{Definition}
\newtheorem{fact}[defn]{Fact}
\newtheorem{cor}[defn]{Corollary}
\newtheorem{propn}[defn]{Proposition}
\newtheorem{concl}[defn]{Conclusion}

\newtheorem{conv-r}[defn]{Conventions and Remarks}

\newtheorem{lemma}[defn]{Lemma}
\newtheorem{obs}[defn]{Observation}

\newcommand{\mcf}{\mathcal{F}}
\newcommand{\mcs}{\mathscr{C}}

\newcommand{\xp}{\mathfrak{p}}
\newcommand{\xt}{\mathfrak{t}}
\newcommand{\thistheoremname}{}
\newtheorem*{genericthm*}{\thistheoremname}
\newenvironment{namedthm*}[1]
  {\renewcommand{\thistheoremname}{#1}%
   \begin{genericthm*}}
  {\end{genericthm*}}

\newcommand{\bc}{\mathbf{c}}
\newcommand{\cf}{\operatorname{cof}}
\newcommand{\dom}{\operatorname{dom}}
\newcommand{\lcf}{\operatorname{lcf}}

\newcommand{\mch}{\mathcal{H}}

\newtheorem{qst}[defn]{Question}

\theoremstyle{definition}

\newtheorem{expl}[defn]{Example}
\newtheorem*{const}{Construction}

\title{Notes on cofinality spectrum problems}

\author{D. Casey and M. Malliaris}

\address{Group in Logic, University of California, Berkeley}
\email{caseydj@berkeley.edu}

\address{Department of Mathematics, University of Chicago} 
\email{mem@math.uchicago.edu}

\date{Version of \today.}

\begin{document}


\maketitle

\vspace{3mm}

These notes are based on Appalachian Set Theory lectures given by M. Malliaris on November 5, 2016 with D. Casey as the 
official scribe.   
The aim of the lectures was to present the setup and some key arguments of  ``Cofinality spectrum problems in model theory, set theory and 
general topology'' by Malliaris and Shelah \cite{MiSh:998}. 


Each section begins with a short abstract. The reader looking for a \emph{very} brief overview may begin with these. The sections 
give more detail, closely following the line of \cite{MiSh:998}. Our aim is to explain in some sense where the proofs come from. 
These notes are complementary to the paper \cite{MiSh:998}; at times we refer there for full details of a definition or proof.   We hope the 
reader will get a sense of the territory around the program of Keisler's order, and also of how much interesting work remains to be done, and 
may be inspired to look into related open problems \cite{MiSh:1069}.  

In these notes, \emph{all languages are countable and all theories are complete} unless otherwise stated. 

\setcounter{tocdepth}{1}
\tableofcontents

\newpage

\section{Dividing lines}

\begin{quotation}
\noindent 
\emph{One of the main aims of model-theoretic classification theory has been to find \emph{dividing lines} among the first order theories, 
and alongside this to prove structure theorems explaining why theories on one side of the line are in some sense simple or tame, while those on the 
other side are in a complementary sense complex or wild. Not all interesting properties are dividing lines, but certainly it is a strong 
recommendation and helps to find sharp theorems.} 
\end{quotation}

\br

Dividing lines are a rather remarkable phenomenon.\footnote{On the thesis of looking for dividing lines, the reader may consult the original 
\emph{Classification Theory} \cite{Sh:a} or the recent discussion in Shelah \cite{Sh:xx} p. 5.} Not so many are known, 
but the few that have been discovered bring great clarity.  
A foundational example for modern model theory has been the dividing line of stability/instability.\footnote{See \emph{Classification Theory} chapters II-III. The Stone space $\mathbf{S}(M)$ of a model $M$ is the set of ultrafilters on the Boolean algebra of $M$-definable sets -- in other words, the set of  complete types over $M$. Call a theory $T$ \emph{stable in $\lambda$} if for all $M \models T$ of size $\lambda$, $|M| = |\ts(M)|$. 
If for \emph{some} $M \models T$ we have $\lambda = |M| < |\ts(M)|$, we call $T$ \emph{unstable in $\lambda$}.  
If $M$ is an algebraically closed field, $|\ts(M)| = |M|$; if $M = (\mathbb{Q}, <)$ there are types for each Dedekind cut, among others, so $|\ts(M)| > |M|$.  
By definition a given $T$ is either stable or unstable in a given $\lambda$, but the remarkable fact is that varying $\lambda$ a gap appears:  Shelah proved that any $T$ is either ``unstable,'' meaning unstable in all $\lambda$, or ``stable,'' meaning stable in all $\lambda$ such that $\lambda^{|T|} = \lambda$ -- ignoring cases where cardinal arithmetic might give a false positive. }  Examples of stable theories include algebraically closed fields of fixed characteristic and  free groups on a fixed finite number of generators; 
examples of unstable theories include 
random graphs, linear orders, and real closed fields.  The stable theories have by now a beautiful structure theory, but most theories, 
including many of great interest, are unstable.  What is the right approach to classifying them?

\section{Keisler's order}
\begin{quotation}
\noindent 
\emph{An old open problem about saturation of regular ultrapowers may be understood as giving a powerful framework for searching for dividing lines (including stability) in a systematic way.}
\end{quotation}
\br


\begin{defn} 
For $\de$ an ultrafilter on $\lambda$, say that ``$\de$ saturates $M$'' if $N = M^\lambda/\de$ is $\lambda^+$-saturated.\footnote{Recall that this means that any type over any $A \in {[N]^{\leq \lambda}}$ consistent with $N$ is realized in $N$, or if you prefer (recalling $\ml$ is countable) any collection of $\lambda$ definable subsets of $N$ with the finite intersection property has a nonempty intersection in $N$.}
\end{defn}

A priori, the amount of saturation of $N$ is a function both of the ultrafilter $\de$ and of the amount of saturation of the model $M$. 
The following lemma of Keisler explains the importance of \emph{regular ultrafilters} for model theorists: the amount of saturation depends on the theory, not the model chosen.  This may be taken as a definition of regular, or see below. 

\begin{propn}[\cite{K67}, Cor. 2.1a] \label{k:prop}
If $\mathcal{D}$ is a regular ultrafilter on $\lambda$ and $M \equiv N$ in a countable language, $M^{\lambda} / \mathcal{D}$ is $\lambda^+$-saturated if and only if $N^{\lambda} / \mathcal{D}$ is {$\lambda^+$-saturated}.
\end{propn}

Thus, if $\de$ is regular, we may simply say ``$\de$ saturates $T$'' if $\de$ saturates some, equivalently every, model of $T$.  
This suggests a means of comparing theories according to the ``likelihood'' that their regular ultrapowers are saturated, 
which is made precise in Keisler's order \cite{K67}. 

\begin{defn}[Keisler's order, 1967] Let $T_1, T_2$ be complete countable theories. 
\[ T_1 \trianglelefteq T_2 \]
if for every regular ultrafilter $\mathcal{D}$, if $\mathcal{D}$ saturates $T_2$ then  $\mathcal{D}$ saturates $T_1$.
\end{defn}

Regularity may also be defined directly: it is a kind of strong incompleteness. The ultrafilter $\de$ on $I$  is regular if there is $\{ A_\alpha : \alpha < |I|\} \subseteq \de$, 
called a regularizing family, such that the intersection of any countably many distinct elements of the family is empty.\footnote{Equivalently, for every $i \in I$,  $| \{ {\alpha} \mid i \in A_{\alpha} \} | < \aleph_0$.  Regular ultrafilters exist on every cardinal, see \cite{ck} Prop. 4.3.5 p. 249; indeed Donder proved that consistently all ultrafilters are regular.}

\begin{ex}
Prove, or read Keisler's proof of, Proposition \ref{k:prop}.
\end{ex}

There is much to say about this extremely interesting order, which frames the proof we are studying.\footnote{The current picture of the structure of Keisler's order can be found in \cite{MiSh:1050}; \cite{mm5} and \cite[\S 2]{MiSh:1030} 
give more information. For more on the set-theoretic side of Keisler's order, see  \cite{MiSh:996}. 
Note that in the instance we are studying here, it was model theory that 
influenced set theory, but there is potential for much in the other direction -- see for instance \cite{MiSh:1050}. }
Briefly, Keisler's order studies the relative complexity of theories according to the difficulty of ensuring their regular 
ultrapowers are saturated. It was known since Shelah 1978 \cite[VI.5]{Sh:a} that the union of the first two classes in $\tlf$ 
is precisely the stable theories; in other words, Keisler's order gives an 
independent way to detect the dividing line at stability. 

So, potentially, Keisler's order may give a way to 
look systematically for dividing lines among the unstable theories.  This was a thesis motivating Malliaris and more recently 
Malliaris and Shelah in very productive joint work on the order. But this is not just of model theoretic interest. 
In this framework ultrafilters are leveraged to compare the complexity  of theories, but in parallel, theories 
reflect the complexity of ultrafilters.

The starting point for the paper \cite{MiSh:998}, the third joint paper of Malliaris and Shelah on Keisler's order, 
was a question about a model-theoretic sufficient condition for 
maximality in Keisler's order.

\section{The maximum class}  \label{s:max}

\begin{quotation}
\noindent 
\emph{Already in 1967, Keisler had shown that his order has a maximum class, which may be characterized set theoretically by a property of ultrafilters, 
called \emph{good}.}
\end{quotation}

\br

In his 1967 paper Keisler had proved the order $\tlf$ had a maximum class, essentially by the following argument. 
(a)  Good regular ultrafilters saturate any [countable] $T$. (b) Some theories, e.g. Peano arithmetic, are saturated by a given regular ultrafilter only if it is good. 
(c) \emph{Thus} the set 
$$\{ T : \mbox{ if $T$ is saturated by a regular ultrafilter $\de$ then $\de$ must be good }\}$$ is nonempty and is the 
maximal, in fact maximum, class in $\tlf$.

\begin{defn}
We say the ultrafilter $\de$ on $\lambda$ is \emph{good} if every monotonic function $f: [\lambda]^{<\aleph_0} \rightarrow \de$ has a multiplicative refinement, 
meaning that: if $f$ satisfies $u \subseteq v \implies f(u) \supseteq f(v)$  for all finite $u,v$, then there is $g: [\lambda]^{<\aleph_0} \rightarrow \de$ 
such that for all finite $u$, $g(u) \subseteq f(u)$, and for all finite $u,v$, $g(u \cup v) = g(u) \cap g(v)$.
\end{defn}

The existence of good regular ultrafilters on any infinite 
$\lambda$ was proved by Keisler \cite{keisler-1} under GCH and unconditionally by Kunen about a decade later \cite{kunen}.

Notation: when dealing with an ultrapower $N = M^I/\de$, we first fix a lifting 
$M^I/\de \rightarrow M^I$, so that for any $a \in N$ and $t \in I$ the projection $a[t]$ is well defined. When $\bar{a} = \langle a_1,\dots, a_n \rangle$ is a tuple in $N$, write $\bar{a}[t]$ for $\langle a_1[t], \dots, a_n[t]\rangle$.  (The choice of lifting must be made in advance, but will not matter.)

Keisler's argument rests on the following correspondence. 
Let $N = M^I/\de$ be a regular ultrapower, $|I| = \lambda$, 
$p(x) = \{ \vp_\alpha(x;\bar{a}_\alpha) : \alpha < \lambda \rangle$ a type, and {$\{ X_\alpha : \alpha < \lambda \}$} a regularizing family. 
For each finite $u \subseteq \lambda$, consider the map $f$ given by 
\[ u \mapsto \{ t \in I ~: ~\exists x \bigwedge_{\alpha \in u}\vp_\alpha(x;\bar{a}_\alpha[t]) \} \cap \bigcap_{\alpha \in u} X_\alpha. \] 
Then $f$ is monotonic, and by \lost theorem, its range is a subset of $\de$.  

\begin{ex} \label{ex1}
Prove that $p$ is realized in $N$ iff $f$ has a multiplicative refinement.\footnote{This is spelled out in e.g. \cite{mm4} Observation 3.10. Note the  correspondence: a theory is saturated precisely when all the functions which correspond to its types in the manner just described have multiplicative refinements. A priori, this may be fewer than all monotonic functions.}
\end{ex}

As $M$ was arbitrary, Exercise \ref{ex1} explains item (a) from the beginning of the section.  
To prove (b), one needs a theory such that for each given monotonic function, it's possible to build a type 
whose projections have precisely the pattern of incidence that function records; any theory with sufficient coding will do.\footnote{In Keisler's parlance, any 
\emph{versatile formula} will suffice, see \cite[Theorem 6.1]{K67}.}

No model-theoretic characterization of the maximal class was known, i.e. no model-theoretic 
necessary and sufficient condition for theories to be maximal was known; the above discussion suggests the following 
general approach to this (still open) problem.\footnote{On maximality, look at the title of \cite{K67}. Nor is it so 
easy to meaningfully weaken goodness (see for instance Dow's notion of OK). The introduction to \cite{MiSh:1070} goes into this 
question in some detail.) }  
We might consider a reasonably complex, but not obviously already maximal, property $X$ of theories; 
find a property $Y$ of regular ultrafilters which is necessary to saturate such theories; and 
ask if any regular ultrafilter with property $Y$ must be good.
If this can be carried out, both answers are useful; \emph{yes} says $X$-theories are maximal, and \emph{no} gives an interesting weakening of goodness.

\section{A translation}

\begin{quotation}
\noindent 
\emph{The starting point was the question of whether the model-theoretic property $SOP_2$ was sufficient for a theory to be maximal in Keisler's order.  It turns out this can be translated into a question about orders and trees in ultrapowers.}
\end{quotation}

\br
Work on \cite{MiSh:998} began from the question of whether a model-theoretic property\footnote{For completeness, 
\emph{we say the theory $T$ has \emph{$SOP_2$} if for some formula $\vp(\bar{x};\bar{y})$, in some $M \models T$, 
there exist parameters ${\{ \bar{a}_\eta : \eta \in {^{\omega>}2} \}}$ such that for any $n<\omega$ and any distinct 
$\eta_1, \dots, \eta_n \in {^{\omega>}2}$, the set
$ \{ \vp(\bar{x};\bar{a}_{\eta_1}),\dots, \vp(\bar{x};\bar{a}_{\eta_n}) \} $
is consistent if and only if the elements $\eta_1,\dots, \eta_n$ lie on a single branch $($i.e. iff there is $\rho \in {^\omega 2}$
such that $\eta_i \tlf \rho$ for $i=1,\dots,n$$)$.}
An example of $SOP_2$ is given by the formula $\vp(x;y_1,y_2) = y_1 > x > y_2$  in any theory of linear order, but there are more interesting  
examples, such as in the triangle-free random graph.   Why $SOP_2$? See \cite{MiSh:998} 
Discussion 11.12.} called $SOP_2$ was sufficient 
to imply maximality in Keisler's order. In \cite[Lemma 11.6]{MiSh:998}, historically an early part of the paper, a necessary condition for a regular ultrafilter on $\lambda$ 
to saturate any theory with $SOP_2$ was found, called \emph{treetops}. 

\begin{defn} \label{d:treetops}
Say that the regular ultrafilter $\mathcal{D}$ on $\lambda$ has \emph{$\kappa$-treetops} if whenever 
$M = (\mathcal{T}, \tlf)$ is a tree\footnote{By  `tree' we mean here 
a set $\mct$ given with a partial order $\tlf$ such that the set of predecessors of any node is well ordered. 
In slight abuse of notation, we'll eventually use `tree' to refer to ultrapowers of trees, and later, elements of $\mct(\cs)$.} 
and $N = M^\lambda/\de$ its ultrapower, any strictly $\tlf$-increasing sequence of cofinality $<\kappa$ has an upper bound. $\de$ has 
\emph{treetops} if it has $\lambda^+$-treetops. 
\end{defn}

Recalling the end of \S 3, this leads one to ask: 

\begin{qst} \label{q23} 
Let $\de$ be a regular ultrafilter. 
If $\mathcal{D}$ has treetops, is $\mathcal{D}$ good?
\end{qst}

Meanwhile, we can also translate ``good'' using model theory. 
Shelah had proved in 1978 that the theory of linear order is in the maximal Keisler class \cite[Theorem VI.2.6]{Sh:a}. 
This suggests that we may try to measure an ultrafilter's goodness by its effect on a model of linear order. 
Towards this, let us define the \emph{cut spectrum} of an ultrafilter on $\lambda$. 
We will say a linear order has a $(\kappa_1, \kappa_2)$-pre-cut when there is a strictly increasing 
sequence $\langle a_\alpha : \alpha < \kappa_1 \rangle$ and a strictly decreasing sequence 
$\langle b_\beta : \beta < \kappa_2 \rangle$ with $a_\alpha < b_\beta$ for all $\alpha < \kappa_1$ and all $\beta < \kappa_2$.  
A pre-cut which is not filled is called a cut.

\begin{defn} \label{d:cut}
Let $\de$ be an ultrafilter on $\lambda$. Define its \emph{cut spectrum} as
\[ \mc(\de) = \{ (\kappa_1, \kappa_2) : \kappa_1, \kappa_2 \mbox{ are regular and } \leq \lambda \mbox{ and }(\omega, <)^\lambda/\de
\mbox{ has a $(\kappa_1, \kappa_2)$-cut } \}.\] 
\end{defn}

When the ultrafilter $\de$ is regular, the cut spectrum $\mc(\de)$ has several key properties.  First, 
we may replace $(\omega, <)$ in Definition \ref{d:cut} with any other infinite linear order (e.g. $\mathbb{Q}, \omega_1,...$) and 
the cut spectrum will not change (see the appendix to \cite{MiSh:1069} for a proof).  
Second, the cut spectrum captures saturation:\footnote{Exercise for the reader: prove that regularity of $\de$ ensures the only relevant  
omitted types in a regular ultrapower of $(\omega, <)$ arise as $(\kappa_1, \kappa_2)$-cuts where $\kappa_1, \kappa_2$ are both infinite.} 
$\mc(\de) = \emptyset$ if and only if $\de$ saturates 
$(\omega, <)^\lambda/\de$. 
Finally, the fact that linear order is in the maximal Keisler class means that $\de$ saturates $(\omega, <)$ if and only if $\de$ is good. 

\begin{concl} 
For a regular ultrafilter $\de$, the following are equivalent:  $\mc(\de) = \emptyset$, $\de$ saturates $(\omega, <)$  $($or any other infinite linear order$)$, $\de$ is good. 
\end{concl}

So our Question \ref{q23} becomes:

\begin{qst} \label{q4.6} 
Let $\de$ be a regular ultrafilter. 
If $\de$ has treetops, is $\mc(\de) = \emptyset$?
\end{qst}

\section{Two cardinals} \label{s:two} 

\begin{quotation}
\noindent 
\emph{We may focus the question about orders and trees by defining two cardinal 
invariants of an ultrafilter, $\xp_\de$ and $\xt_\de$.}
\end{quotation}

\br
We are comparing where a path through a tree fails to have an upper bound, 
and the appearance of a cut.  To make comparison easier, we may define the following.   
For now the names are simply suggestive.

\begin{defns} \label{d:pd-td} For $\mathcal{D}$ a regular ultrafilter on $\lambda$:
\begin{itemize}
\item $\mathfrak{p}_{\mathcal{D}}$ is the minimum $\kappa$ such that there is in $(\omega, <)^\lambda/\de$ a $(\kappa_1, \kappa_2)$-cut with $\kappa = \kappa_1 + \kappa_2$. 

\item $\mathfrak{t}_{\mathcal{D}}$ is the minimum $\kappa$ such that for some tree  $(\mathcal{T}, \tlf)$, there is in 
${(\mct, \tlf)^\lambda/\de}$ a strictly $\tlf$-increasing $\kappa$-indexed sequence with no upper bound. 

\end{itemize}
\end{defns}



Observe that in this language Question \ref{q4.6} becomes:

  
\begin{qst} \label{q5.2}
Can $\mathfrak{p}_{\mathcal{D}} < \mathfrak{t}_{\mathcal{D}}$?
\end{qst}

\section{Warm-up} \label{s:warmup}

\begin{quotation}
\noindent 
\emph{Test your understanding of the two
 fundamental theorems of ultraproducts.}
\end{quotation}

\br


\begin{theorem-x}[\L o\'{s}'s Theorem] Let $\varphi(\overline{a})$ be a first-order sentence with parameters. 
Then $M^{\lambda} / \mathcal{D} \vDash \varphi(\overline{a})$ if and only if $\{ t < \lambda \mid M \vDash \varphi(\overline{a}[t]) \} \in \mathcal{D}$.
\end{theorem-x}

\begin{expl}
Every bounded, nonempty, definable subset of $(\omega, <)^{\lambda} / \mathcal{D}$ has a minimum element and a maximum element.  
\end{expl}

\begin{theorem-x}[Ultrapowers commute with reducts]
Suppose $\mathscr{L} \subseteq \mathscr{L}_*$ are languages and  $M_*$ is a $\mathscr{L}_*$-structure. Then: 
$$  \left( M_*^{\lambda} / \mathcal{D} \right) \restriction_{\mathcal{L}} ~  \cong ~ \left( M_* \restriction_{\mathcal{L}} \right)^\lambda / \mathcal{D}.$$
\end{theorem-x}

\begin{ex} \label{ex6.2}
Let $M$ be a countably infinite model with a binary relation $E$, interpreted as an equivalence relation with infinitely many 
infinite classes. What do its ultrapowers look like? In particular, can the classes be of different sizes?\footnote{Consider first expanding  
the model to add a family of bijections between the classes. Does it matter  
whether we forget these bijections before {or} after taking the ultrapower?} 
\end{ex}

We see from this that ultrapowers respect potential as well as actual structure.\footnote{So where is there room for any variation in 
the structure of e.g. regular ultrapowers? What can fly under the radar of such expansions? A major answer is: pseudofinite structure. 
In \ref{ex6.2}, if $E$ has 
a class of size $n$ for each $n$, then the ultrapower will have many 
infinite classes, but there is no a priori reason they should have the same size: we can't a priori play the same game with internal bijections, 
since two given infinite classes may have different finite sizes almost everywhere.}

\section{Tree notation}

\begin{quotation}
\noindent 
\emph{We fix some notation for dealing with trees.}
\end{quotation}
\br

In many of our arguments, we'll have a discrete linear order $X$ (say, $\mathbb{N}$) 
and we will be interested in a tree $\mathcal{T}$ whose elements are, say, functions from an initial segment of $X$ into a finite Cartesian power $X^k$ 
(possibly satisfying some additional conditions, like monotonicity in one or more coordinates),  
where the tree-order $\trianglelefteq$ is the partial order given by initial segment. 
For any such $c \in \mathcal{T}$ and $n \in \operatorname{dom}(c)$, ~ $c(n)$ is a $k$-tuple. Write $c(n,i)$ for its first coordinate, ... 
$c(n,k-1)$ for its last coordinate.  In cases of interest, $\maxdom(c) = \lgn(c)-1$ will be well defined, and definable.  We will denote concatenation by $c^{\smallfrown} \langle a_0, ..., a_{k-1} \rangle$, i.e. this denotes the partial function that agrees with $c$ on $\dom(c)$ and equals $\langle a_0, ..., a_{k-1} \rangle$ on the successor of $\maxdom(c)$. 

Here are two examples. First, consider ``$\mathcal{T}_1$ is the tree of finite sequences of pairs of natural numbers, strictly increasing in each coordinate,'' 
i.e., $\mathcal{T}_1 = \{ c : c$ is a function from an initial segment of $\mathbb{N}$ to $\mathbb{N} \times \mathbb{N}$ and 
$i < j \leq \maxdom(c)$, $c(i,0) < c(j,0)$ and $c(i,1) < c(j,1)$ $\}$.  Second, consider: ``$\mathcal{T}_2$ is the tree of finite sequences $c$ of pairs of natural numbers such that for all $i < j \leq \max\operatorname{dom}(c)$, 
$c(i,0) < c(j,0) < c(j,1) < c(i,1)$.''  Notice that paths through $\mct_2$ correspond to finite sequences of concentric intervals 
in $\mathbb{N}$.

\section{Symmetric cuts} \label{s:sym}

\begin{quotation}
\noindent 
\emph{As our first evidence that treetops have some control over cuts, we prove there is no symmetric cut below $\xt_\de$.} 
\end{quotation}

\br

\begin{lemma}[c.f. \cite{MiSh:998}, Lemma 2.2] \label{symlem}
If $\mathcal{D}$ has $\lambda^+$-treetops and $\kappa < \lambda^+$ is regular then $(\kappa, \kappa) \notin \mathscr{C}(\mathcal{D})$. 
\end{lemma}

\begin{proof} 
Let $M = (\omega, <)$ and let $N = M^\lambda/\de$.  
Suppose for a contradiction that in $N$ we had some symmetric 
[unfilled] cut $(\bar{a}, \bar{b}) = (\langle a_\alpha: \alpha <\kappa \rangle, \langle b_\alpha : \alpha < \kappa \rangle)$. 

Let $(\mct, \tlf)$ be the tree whose elements are finite sequences of pairs of natural numbers (\emph{i.e.} 
functions from some finite initial segment of $\omega$ to $\omega \times \omega$) partially ordered by initial segment. 
Expand $M$ to a model $M^+$ in which this tree is definable and in which the following are also uniformly definable for 
$c \in \mct$: 
\begin{itemize}
\item the length function $\lgn(c) = \maxdom(c)-1$,
\item for each $n \leq \maxdom(c)$, the evaluation function 
$c(n)$,
\item for each $n \leq \maxdom(c)$, the two projection functions $c(n,0)$ and $c(n,1)$,
 \item $\tlf$, the partial order on $\mct$ given by initial segment.
\end{itemize}
(For example, we could let $M^+$ be the hereditarily countable sets 
$(\mch(\omega_1), \epsilon)$ and identify $M$ with $\omega$ in $M^+$.)

Since ultrapowers commute with reducts, there is a parallel expansion of the ultrapower $N$ to $N^+$, in which the parallel tree is 
definable; elements of $\mct^{N^+}$ will have as their domain some initial segment of nonstandard integers.  

Let $\vp(x)$ define the sub-tree of $\mct$ whose branches describe concentric pre-cuts: 
$\{ x \in \mct :$  for all $i < j \leq \maxdom(x)$, $x(i,0) < x(j,0) < x(j,1) < x(i,1) \}$. 
Let $\mct^{N^+}_* = (\mct^{N^+}_*, \tlf)$ be defined by $\vp$ in $N$.
\footnote{We should write $\tlf^{N^+}$, but have not done so for readability.} \footnote{Why make this a two-step process -- first defining 
$\mct$, then defining $\mct_*$? This is mainly expositional: soon, in the general definition of CSPs which this proof motivates, it will be simplest to 
just assume that trees of functions from orders to themselves exist, and then observe that if long sequences in 
these trees have upper bounds, the same is true of non-trivial definable subtrees.} 
Observe that by \S \ref{s:warmup}, we have: 
\begin{enumerate}
\item[(i)] (\emph{pseudofiniteness}) ~every $c \in \mct^{N^+}_*$ has a maximal element of its domain, and moreover, 
every nonempty definable subset of $\{ n : n \leq \maxdom(c) \}$ -- in fact, every bounded nonempty definable 
subset of $N$ -- has a greatest and least element, where definable means definable in the expanded model $N^+$, with parameters. 
\item[(ii)] (\emph{concatenation})  ~if $c \in \mct^{N^+}_*$ and $n = \maxdom(c)$ and $a,b \in \omega^{N_*}$ and 
$c(n,0) < a < b < c(n,1)$, then $c^\smallfrown \langle a, b \rangle \in \mct^{N^+}_*$. 
\end{enumerate}
Now let's build our cut $(\bar{a}, \bar{b})$ into a path through the tree. That is, by induction on 
$\alpha < \kappa$ we build a path $\langle c_\alpha : \alpha < \kappa \rangle$ through $\mct^{N^+}_*$ such that for all 
$\alpha < \kappa$, writing $n_\alpha = \max\operatorname{dom}(c_\alpha)$, we have
\[ c_\alpha(n_\alpha, 0) = a_\alpha \textrm{ and } c_\alpha(n_\alpha, 1) = b_\alpha. \]

Case $\alpha = 0$. Take $c_0 = \langle a_0, b_0 \rangle$, using concatenation, as $\emptyset \in \mct^{N^+}_*$. 

Case $\alpha = \beta + 1$. Take $c_\alpha = {c_\beta}~^\smallfrown \langle a_\alpha, b_\alpha \rangle$. 

Case $\alpha$ limit: $\langle c_\beta : \beta < \alpha \rangle$ is a path through $\mct^{N^+}_*$ 
of cofinality $\cf(\alpha) < \kappa < \lambda^+$. 
Apply treetops to find an upper bound $c_* \in \mct^{N^+}$ such that $\beta < \alpha \implies c_\beta \trianglelefteq c_*$. 
Letting $n_* = \max\operatorname{dom}(c_*)$ we have that $\beta < \alpha$ implies 
\[ c_\beta(n_\beta, 0) = c_*(n_\beta, 0) < c_*(n_*, 0) < c_*(n_*, 1) < c_*(n_\beta, 1) =  c_\beta(n_\beta, 1). \]
But in terms of the cut, maybe we overshot: maybe $c_*(n_*, 1) < a_{\alpha}$ or $c_*(n_*, 0) > b_{\alpha}$, 
which would block our next inductive step. Recalling (i), the set 
\[ \{ n \leq n_* :  c_*(n,0) < a_\alpha \textrm{ and } c_*(n,1) > b_\alpha \}.\]
has a maximal element $n_{**}$.  
Note that the definition of $n_{**}$ guarantees that $n_{\beta} < n_{**}$ for all $\beta < \alpha$.
Let $c_\alpha = c_* \upharpoonright_{n_{**}}~^\smallfrown \langle a_\alpha, b_\alpha \rangle$.  

Having constructed $\bar{c} = \langle c_\alpha : \alpha < \kappa \rangle$, we apply treetops one more time 
to find $c_\star$ above $\bar{c}$. Then letting $n_\star = \max\operatorname{dom}(c_\star)$, we see that $\alpha < \kappa$ implies
\[ a_\alpha = c_\alpha(n_\alpha, 0) < c_\star(n_\star, 0) < c_\star(n_\star, 1) < c_\alpha(n_\alpha, 1) = b_\alpha \]
so $c_\star(n_\star,\ell)$ for $\ell = 0, 1$ fill the cut.  This contradiction completes the proof.\footnote{Note that the wording of this as a proof by contradiction is not necessary; we are essentially showing that any such symmetric pre-cut is filled.  Note also the structural information given by the proof. For example, it tells us that given any $(\kappa, \kappa)$-pre-cut $(\bar{a}, \bar{b})$, there is an internal map taking the sequence $\bar{a}$ to the sequence $\bar{b}$.}
\end{proof}

\section{Our true context}  \label{s:CSP}

\begin{quotation}
\noindent 
\emph{Axiomatizing the basic properties we used in the last proof, we arrive to the natural setting for our arguments, called \emph{cofinality 
spectrum problems}.
To each cofinality spectrum problem $\cs$, we associate cardinal invariants $\xp_\cs$ and $\xt_\cs$ and a  
cut spectrum $\mcs(\cs, \xt_\cs)$.}
\end{quotation}

\br

The previous proof suggests the potential strength of the connection between treetops and cuts in ultrapowers of linear orders.  
But notice that this proof used only a few facts about ultrapowers. The proof would work for any elementary 
pair of models $M \preceq N$ in which we were given formulas $\Delta$ defining discrete linear orders in $M$ (so also in $N$) 
provided that: \textit{first}, we could expand the models in parallel to $M^+ \preceq N^+$ with enough set theory to 
define appropriate trees, and \textit{second}, that each order in $N$ defined by a $\Delta$-formula is 
\emph{pseudofinite}, i.e. every bounded, nonempty, definable (even in the expanded language) subset has a first and last element.

This is made formal in a central definition of the paper, ``$\cs = (M, N, M^+, N^+, \Delta)$ is a cofinality spectrum problem,''
see  \cite{MiSh:998} 2.3-2.5.  Although this definition is longer than that of an regular ultrapower, it is in some sense simpler: 
it's just a basic set of requirements on a pair of models, with all our assumptions displayed.

We summarize to fix notation, but encourage the reader to 
 read the full definition in \cite{MiSh:998} before continuing.  
A CSP $\cs$ has a set of \emph{orders} $\ord(\cs)$ and a set of \emph{trees} $\tr(\cs)$.

\begin{enumerate}
\item The orders $\ord(\cs)$:
\begin{enumerate}
\item For any $\vp \in \Delta$ and $\overline{c} \in {^{\ell(z)}N}$, $\varphi(\overline{x}, \overline{y}, \overline{c})$ gives a discrete linear order ``$\leq$'' on the set ``$X$'' $= \{ \vp(\bar{a}, \bar{a}, \bar{c}) : \bar{a} \in {^{\ell(x)} N } \}$.  
Each such order is pseudofinite in $N^+$,  meaning every bounded, nonempty, definable subset has a maximum and minimum element. 

\item Formally, the data of an order $\ma \in \ord(\cs)$  is given by its defining formula and parameter, 
$\vp_\ma(\bar{x}, \bar{y}, \bar{c}_\ma)$, along with the choice of a designated element $d_\ma$ in the ordered set, see below.  
We abbreviate the ordered set as $X_\ma$ and the discrete linear ordering on it as $\leq_\ma$. Note we can define $0_\ma$, 
the least element of $X_\ma$, and the successor and predecessor functions. Call $\ma$ \emph{nontrivial} if $d_\ma$ isn't a finite successor of $0_\ma$.

\item $\ord(\cs)$ is closed under Cartesian products, i.e. for each $\ma \in \ord(\cs)$ there is at least one 
$\mb = \ord(\cs)$ with $X_\mb = X_\ma \times X_\ma$ (we may write $\mb = \ma \times \ma$ when this holds). 
The coordinate projections of such products are definable. 
For at least one nontrivial $\ma$, there is $\mb = \ma \times \ma$ whose ordering interacts reasonably with the order on each factor, 
e.g. it arises from the G\"odel pairing function. 
\end{enumerate}

\noindent Summarizing notation, an order $\ma \in \ord(\cs)$ is 
$\ma = (X_\ma, \leq_\ma, \vp_\ma, \bar{c}_\ma, d_\ma)$.

\br
\item The trees $\tr(\cs)$: 

\begin{enumerate}
\item For each $\ma \in \ord(\cs)$, there is an associated definable tree $\mct_\ma = {(\mct_\ma, \tlf_\ma)}$ consisting of partial functions from the order to itself, definably partially ordered by initial segment.  
\item For each tree $\mct_\ma \in \tr(\cs)$, the following are also uniformly definable: the length of any $b \in \mct_\ma$ 
and its value at any point in its domain; and if $\lgn(b) <_\ma d_\ma$, \emph{see} (1)(b), we have definable {concatenation}, ensuring 
that $b^\smallfrown\langle a \rangle \in \mct_\ma$ 
for any $a \in X_\ma$. 


\end{enumerate}
\noindent Summarizing notation,  
 a tree $\mct \in \ord(\cs)$ is $(\mct_\ma, \tlf_\ma)$ 
for some $\ma \in \ord(\cs)$.

\end{enumerate}

\br
\noindent
Recall that we aim to analyze how the appearance of cuts in one of our distinguished orders 
relate to the existence of unbounded paths in the distinguished trees. 

\begin{defn} For $\mathbf{s}$ a cofinality spectrum problem $($CSP$)$:
\begin{itemize}
\item $\mathfrak{p}_{\mathbf{s}}$ is the minimum $\kappa$ such that in some order $X_{\mathbf{a}} \in \textrm{Or}(\mathbf{s})$ 
 there is a $(\kappa_1, \kappa_2)$-cut with $\kappa = \kappa_1 + \kappa_2$.
\item $\mathfrak{t}_{\mathbf{s}}$ is the minimum $\kappa$ such that in some tree $\mathcal{T}_{\mathbf{a}} \in \textrm{Tr}(\mathbf{s})$ there is a 
strictly $\tlf$-increasing $\kappa$-indexed sequence with no upper bound. 
\end{itemize}
\end{defn}
\noindent
Recalling Definition \ref{d:cut}, define the \emph{cut spectrum}: 
$$\mathscr{C}(\mathbf{s}) = \{(\kappa_1, \kappa_2) \mid \textrm{ there is a } (\kappa_1, \kappa_2)\textrm{-cut in some }X_{\mathbf{a}} \}.$$
When the size of the cut is important, we use the notation 
$$\mathscr{C}(\mathbf{s}, \mu) = \{(\kappa_1, \kappa_2) \mid \textrm{ there is a } (\kappa_1, \kappa_2)\textrm{-cut in some }X_{\mathbf{a}}   
\mbox{ and } \kappa_1 + \kappa_2 < \mu \}.$$
The most important case for us will be $\mcs(\cs, \xt_\cs)$, i.e. the ``cuts below treetops''.

In the case of regular ultrapowers (a main example of CSPs), these definitions correspond to the earlier ones.\footnote{We may build a 
CSP from an ultrapower of linear order just as in the last proof: e.g. let  
$\cs$ be formed from $M = (\omega, <)$, its regular ultrapower $M^\lambda/\de$, their expansions to 
models of sufficient set theory, and the set of formulas defining linear orders on initial segments of $\omega$ or its finite Cartesian products 
(with the order given by G\"odel coding, and in each case, $d_\ma = \max X_\ma$); 
$\xt_\cs = \xt_\de$, and $\xp_\cs = \xp_\de$. $\mcs(\de)$ becomes $\mcs(\cs, \lambda^+)$.  
In this setup, note that ``$\de$ has treetops'' just means $\xt_\cs \geq \lambda^+$.}  
However, the move to CSPs has increased our range; it includes models of Peano arithmetic \cite{MiSh:1051}, and more, as we'll see next.

\section{$\mathfrak{p}$ and $\mathfrak{t}$} \label{s:p-t}

\begin{quotation}
\noindent 
\emph{Assuming $\xp < \xt$, we construct a cofinality spectrum problem $\cs$, built from a generic ultrapower, in which  $\xp_\cs \leq \xp < \xt \leq \xt_\cs$, where $\xp$ is the pseudointersection number and $\xt$ is the tower number. }
\end{quotation}

\br

In this section we will \emph{assume} that $\xp < \xt$, and 
build a cofinality spectrum problem $\cs$ for which $\xp_\cs \leq \xp < \xt \leq \xt_\cs$. This will mean:  
\emph{if} we can prove, in ZFC, that for any cofinality spectrum problem $\cs$, $\xt_\cs \leq \xp_\cs$, 
\emph{then} $\xp = \xt$.  We follow \cite{MiSh:998} \S 14. 
(We won't repeat here the long history of measuring the continuum by cardinal invariants. 
A brief introduction to the problem of $\xp$ and $\xt$ may be found in \cite{MiSh:998} \S 1.)

There were earlier
 connections of $\xp < \xt$ to cuts in orders which will show the way.  First recall the definitions. 
Let $A \subseteq^* B$ mean that $A \setminus B$ is finite.  
Given a family $\mcf \subseteq [\mathbb{N}]^{\aleph_0}$, we say that 
$\mcf$ has a \emph{pseudointersection} if there is $A \in [\mathbb{N}]^{\aleph_0}$ such that $A \subseteq^* B$ for all $B \in \mcf$. 
We say that $\mcf$ has the \textit{strong finite intersection property} (SFIP) if any nonempty finite subfamily has an infinite intersection. 

\begin{defns} \emph{ }
\begin{itemize}
\item The \emph{pseudointersection number} $\mathfrak{p}$ is the smallest size of a family $\mcf \subseteq [\mathbb{N}]^{\aleph_0}$ 
with SFIP but no pseudointersection. 
\item The \emph{tower number} $\mathfrak{t}$ is the smallest size of a family $\mcf \subseteq [\mathbb{N}]^{\aleph_0}$ 
which is a tower $($ i.e. linearly ordered by ${\supseteq^*}$ and no pseudointersection$)$.
\end{itemize}
\end{defns}

It is immediate that $\mathfrak{p} \leq \mathfrak{t}$ as being linearly ordered by $\supseteq^*$ implies that every finite subfamily has infinite intersection.  
Rothberger proved in 1948 \cite{Roth48} that if $\xp = \aleph_1$, $\xp = \xt$, and this begs the question of whether $\xp = \xt$. 

By Rothberger's result just quoted, when assuming $\xp < \xt$, we 
can assume $\xp > \aleph_1$.
To connect to cuts, we will need a definition and a theorem from Shelah \cite{Sh:885}. 
Given $f, g \in {^{\omega}\omega}$, we say $f <^* g$ if $f(n) < g(n)$ for all but finitely many $n \in \omega$. 

\begin{defn}[Peculiar cut]
$(\langle g_{\alpha} : \alpha < \kappa_1 \rangle, \langle f_{\beta} : \beta < \kappa_2 \rangle )$ is a \emph{$(\kappa_1, \kappa_2)$-peculiar cut} in $(^{\omega}\omega, <^*)$ when: 
\begin{enumerate}[i)]
\item For any $\alpha < \alpha' < \kappa_1$ and $\beta < \beta' < \kappa_2$, $g_{\alpha} <^* g_{\alpha'} <^* f_{\beta'} <^* f_{\beta}$.
\item For all $h \in {^{\omega}\omega}$:
$$\textrm{If }g_{\alpha} \leq^* h \textrm{ for all } \alpha < \kappa_1 \textrm{, then there is some } \beta < \kappa_2 \textrm{ such that } f_{\beta} \leq^* h.$$
$$\textrm{If }f_{\beta} \geq^* h \textrm{ for all } \beta < \kappa_2 \textrm{, then there is some } \alpha < \kappa_1 \textrm{ such that } g_{\alpha} \geq^* h.$$
\end{enumerate}
\end{defn}

\begin{theorem}[Shelah \cite{Sh:885}] \label{t:pc} If $\mathfrak{p} < \mathfrak{t}$, then for some regular $\kappa$ with $\aleph_1 \leq \kappa < \mathfrak{p}$, there is a $(\kappa, \mathfrak{p})$-peculiar cut in $(^{\omega}\omega, <^*)$.
\end{theorem}
\noindent(In a posteriori wisdom, this relates to asymmetric cuts.)

\vspace{5mm}

\begin{center}
From here through the end of the proof of \ref{prop:ps}: \\
suppose that in a fixed transitive model $\mathbf{V}$ of ZFC, $\xp < \xt$.
\end{center}

\sbr

\begin{const} Assuming $\xp < \xt$, we now build a CSP $\mathbf{s} = (M, N, M_*, N_*, \Delta)$ such that $\mathfrak{t} \leq \mathfrak{t}_\mathbf{s}$ and $\mathfrak{p}_\mathbf{s} \leq \mathfrak{p}$.\\

To begin, we will let $M = M_*$ be a model with enough set theory to satisfy the tree building of a CSP; for concreteness, let $M = M_* = \mathcal{H}(\aleph_1, \in)$, the hereditarily countable sets. 
 For $N = N_*$, we will construct an ultrapower in a forcing extension of our transitive model of ZFC, $\mathbf{V}$.  Let $\mathbf{Q} = ([\mathbb{N}]^{\aleph_0}, \supseteq^*)$ be our forcing notion, with $\mathbf{G}$ some generic subset of $\mathbf{Q}$, forced to be an ultrafilter.  Some important properties of this forcing extension $\mathbf{V}[\mathbf{G}]$ include:
\begin{itemize}
\item $\mathfrak{t}$ is the largest cardinal such that $\mathbf{Q}$ is $\lambda$\textit{-closed}.  So forcing with $\mathbf{Q}$ preserves cofinalities are cardinals up to and including $\mathfrak{t}$.
\item So, as $\mathfrak{p}^{\mathbf{V}} < \mathfrak{t}^{\mathbf{V}} $, we have $\xp^{\vv[\mathbf{G}]} < \xt^{\vv[\mathbf{G}]}$ (in fact, 
$\mathfrak{p}^{\mathbf{V}} = \mathfrak{p}^{\mathbf{V}[\mathbf{G}]}$, $\mathfrak{t}^{\mathbf{V}} = \mathfrak{t}^{\mathbf{V}[\mathbf{G}]}$).
\item There are no new subsets of $\mathbb{N}$.
\item There are no new sequences of length $< \xt$ of elements of $\mathbf{V}$. 
\end{itemize}
Then in $\mathbf{V}[\mathbf{G}]$, we can define the generic ultrapower $N = N_* = M^{\omega} / \mathbf{G}$.  Finally, to complete our CSP, we let $\Delta_{psf}$ consist of all $\varphi(\overline{x}, \overline{y}, \overline{z})$ in $\mathscr{L} = \{ = , \in \}$, such that 
for all $\overline{c} \in M$, $\varphi(\overline{x}, \overline{y}, \overline{c})$ is a finite linear order on $\varphi(\overline{x}, \overline{x}, \overline{c})$ in $M$.  Since we have the analogue of \L o\'{s}'s theorem for $N$,  
it is easy to see that this is indeed a CSP.\footnote{We can specify that $d_\ma$ is always $\max X_\ma$; we have G\"odel coding so 
closure under Cartesian products is easy; we have enough set theory to uniformly define trees.} For the remainder of this section, let 
$\cs$ denote this CSP: 
$$\mathbf{s} = ( M, N, M, N, \Delta_{psf} )$$
i.e., $M = M^+ = (\mch(\aleph_1), \in)$, $N = N^+ = M^\omega/\mathbf{G}$. 
Let's now show $\mathfrak{t} \leq \mathfrak{t}_\mathbf{s}$ and $\mathfrak{p}_\mathbf{s} \leq \mathfrak{p}$.
\end{const}

\begin{propn}
$\mathfrak{t} \leq \mathfrak{t}_\mathbf{s}$.
\end{propn}
\begin{proof}
Consider some regular $\theta < \xt$ and $\ma \in \ord(\cs)$. It will suffice to show that any strictly $\tlf$-increasing $\theta$-indexed 
sequence in the tree $\mct_\ma$ has an upper bound in $\mct$. 

First, a brief reduction: without loss of generality, we may work in the tree $({^{\omega>}\omega}, \tlf)^N$. 
(Sketch:  as $N$ is a generic ultrapower of $M$, we can consider the trees in $Tr(\mathbf{s})$ as arising from ultraproducts of trees in $M$:
$$( X_{\mathbf{a}}, <_{\mathbf{a}}, \mathcal{T}_{\mathbf{a}}, \trianglelefteq_{\mathbf{a}}) = \langle (X_{\mathbf{a}_n}, <_{\mathbf{a}_n}, \mathcal{T}_{\mathbf{a}_n}, \trianglelefteq_{\mathbf{a}_n}) : n < \omega \rangle / \mathbf{G}.$$
Recall that in this CSP, each $(X_{\mathbf{a}_n}, <_{\mathbf{a}_n})$ is a finite linear order in $M$, and each $(\mathcal{T}_{\mathbf{a}_n}, \trianglelefteq_{\mathbf{a}_n})$ is the tree of finite sequences of $X_{\mathbf{a}_n}$.  So we can find an isomorphism between 
each $\mct_{\ma_n}$ and a definable downward closed subset of $({^{\omega>}\omega}, \tlf)^M$. Together these induce an isomorphism of
$\mct_\ma$ on to a definable downward closed subset of $({^{\omega>}\omega}, \tlf)^N$.)

Suppose there were a path $\langle \underaccent{\tilde}{f}_{\alpha} / \underaccent{\tilde}{\mathbf{G}} : \alpha < \theta \rangle$ 
in $({^{\omega>}\omega}, \tlf)^N$ with no upper bound.  Reasoning in $\vv$, there is some $B \in \mathbf{G}$ so that:
$$B \Vdash_{\mathbf{Q}} \mbox{``}\langle \underaccent{\tilde}{f}_{\alpha} / \underaccent{\tilde}{\mathbf{G}} : \alpha < \theta \rangle \textrm{ is a strictly increasing unbounded path in } ({^{\omega>}\omega}, \trianglelefteq)^N \mbox{''}.$$
Since no new sequences of length less than $\mathfrak{t}$ are added, we can also let $B$ force $``\underaccent{\tilde}{f}_{\alpha} = f_{\alpha}"$ for all $\alpha < \theta$ where each $f_\alpha$ is an element of $({^\omega({^{\omega>}\omega})})^\vv$.  Choose\footnote{We can 
do this since $\mathfrak{t}$ is always less than or equal to the \textit{bounding number} $\mathfrak{b}$, the smallest size of a family $\mcf\subseteq{^\omega \omega}$
 such that no $g \in {^\omega \omega}$ eventually dominates all $f \in \mcf$.}
a function $g \in {^\omega \omega}$ such that for each $\alpha < \theta$, for all but finitely many $n \in \omega$, 
\[ g(n) > \sum \{ f_\alpha(n)(i) : i \leq \maxdom(f_\alpha(n)) \} + \maxdom ( f_\alpha(n) ). \]
For each $n \in \omega$, let $s_n$ denote the tree ${^{g(n) \geq}g(n)}$. Then for each $\alpha$, for all but 
finitely many $n$, the finite sequence $f_\alpha(n)$ is an element of the tree $s_n$ (since all the values in its domain and range 
are below $g(n)$). 

Now we look for a potential tower.\footnote{In the notation of \cite{MiSh:998}, let $({^{\omega>}\omega})^{[\nu]} = \{ \eta \in {^{\omega>}\omega} : \nu \leq \eta \}$ denote the ``cone above $\nu$.'' Then $Y_\alpha$  was defined as   
$Y_\alpha = \bigcup \{ \{ n \}  \times \left(  s_n \cap ({^{\omega>}\omega})^{[f_\alpha(n)]}     \right)  : n \in B \}.$
To exactly match the definition of tower given above, fix a bijection $\pi$ from $B \times {^{\omega>}\omega}$ onto $\mathbb{N}$ so each $\pi(Y_\alpha) \in [\mathbb{N}]^{\aleph_0}$.}
For each $\alpha < \theta$ define $Y_\alpha$ (``the disjoint union of the cones in $s_n$ above $f_\alpha(n)$ for $n \in B$'') to be the set
\[  Y_\alpha = \bigcup \{ ~~\{ n \}  \times \{ \eta \in s_n : f_\alpha(n) \tlf \eta \}~ :  n \in B \}. \]
So each $Y_\alpha \subseteq B \times {^{\omega>}\omega}$ and is countably infinite. Moreover, if $\alpha < \beta$ then 
$Y_\beta \subseteq^* Y_\alpha$ since\footnote{$\{ n \in B : f_{\alpha}(n) \not\trianglelefteq f_{\beta}(n) \}$ is finite since $B$ forces $``\underaccent{\tilde}{f}_{\alpha} / \underaccent{\tilde}{\mathbf{G}} \trianglelefteq \underaccent{\tilde}{f}_{\beta} / \underaccent{\tilde}{\mathbf{G}}"$.}
 for all but finitely many $n \in B$, $f_\alpha(n) \tlf f_\beta(n)$. 
Since $\theta < \xt$, the family $\{ Y_\alpha : \alpha < \theta \}$ has a pseudointersection $Z$, and since each $s_n$ is finite, 
$B_1 = \{ n \in B : Z \cap ( \{ n \} \times s_n ) \neq \emptyset \}$ must be infinite. 
For $n \notin B_1$, let $v_n = \langle 0 \rangle$. For $n \in B_1$, let $v_n$ be any element $v$ such that 
$(n, v_n) \in Z \cap ( \{n \} \times s_n )$.  
Then $B_1 \Vdash_{\mathbf{Q}}$ ``$\langle v_n : n \in \omega \rangle/\underaccent{\tilde}{G} $ is an upper bound for 
$\langle f_\alpha/\underaccent{\tilde}{G}: \alpha < \theta \rangle$ in $({^{\omega>}\omega}, \tlf)^N$.''  This completes the proof. 
\end{proof}

\begin{propn} \label{prop:ps}
If $\aleph_1 < \mathfrak{p} < \mathfrak{t}$, then $\mathfrak{p}_\mathbf{s} \leq \mathfrak{p}$.
\end{propn}
\begin{proof}
We show that the $(\kappa, \xp)$-\textit{peculiar cut} that arises assuming $\mathfrak{p} < \mathfrak{t}$ gives us a $(\kappa, \mathfrak{p})$-cut in some $X_{\mathbf{a}}$ in our $N = M^{\omega} / \mathbf{G}$.  Let $(\langle g_{\alpha} : \alpha < \kappa \rangle, \langle f_{\beta} : \beta < \mathfrak{p} \rangle )$ be our peculiar cut and consider first:
$$I = \prod_{n <\omega} [0, f_0(n)] / \mathbf{G} .$$
We have that $I = X_\ma$ for some $\ma \in \ord(\cs)$ by our construction of $\cs$.  Then the peculiar cut forms a \textit{pre-cut} (a potential cut) in $I$.
Suppose that this cut were realized, i.e. suppose there were an infinite $B \in \mathbf{G}$ and $h \in {{^\omega}\omega}$ such that:
$$B \Vdash_{\mathbf{Q}} ``g_{\alpha} / \underaccent{\tilde}{\mathbf{G}} < h / \underaccent{\tilde}{\mathbf{G}} < f_{\beta} / \underaccent{\tilde}{\mathbf{G}} \textrm{ for all } \alpha < \kappa, \beta < \mathfrak{p}" .$$
Then for this infinite\footnote{Note that $B$ may be coinfinite so doesn't contradict existence of Hausdorff gaps.} $B$, we would have both that $B \subseteq^* \{n : g_{\alpha}(n) < h(n) \}$ for all $\alpha < \kappa$  and $B \subseteq^* \{n : f_{\beta}(n) > h(n) \}$ for all $\beta < \mathfrak{p}$.  However, this contradicts the definition of peculiar cut (more precisely, the function 
$h_*$ defined by: $h_*(n)=h(n)$ for $n \in B$ and $h_*(n) = f_0(n)$ for $n \notin B$ is ${\geq^*}~ g_\alpha$ for each $\alpha$ but is 
not $\leq^*$ any of the $f_\beta$'s because $B$ is infinite).  \hfill{\emph{Here ends the assumption that $\xp < \xt$.}}
\end{proof}

\begin{cor} \label{cor2} 
Suppose that we could prove, in ZFC, that 
for \emph{every} cofinality spectrum problem $\cs$, we have that $\mcs(\cs, \xt_\cs) = \emptyset$. It would follow that $\xp = \xt$.
\end{cor}

\section{The central question}

\begin{quotation}
\noindent 
\emph{We arrive to a central problem whose positive solution would answer the two main questions discussed above.}
\end{quotation}

\br

We return to ZFC and to model theory. We'll now work towards answering:

\begin{mqst} \label{main-question}
Let $\cs$ be a cofinality spectrum problem.  Is $\mcs(\cs, \xt_\cs) = \emptyset$?  
\end{mqst}


Put otherwise, Question \ref{main-question} asks: can there be a CSP $\cs$ for which $\xp_\cs < \xt_\cs$? 
(Note: in both main examples of CSPs, it will be the case that $\xp_\cs \leq \xt_\cs$, but 
we don't need this to prove our theorem. It would suffice to show that $\xt_\cs \leq \xp_\cs$. 
We'll often keep track of both $\xp_\cs$ and $\xt_\cs$ in our hypotheses as we go along.\footnote{When are they equal? 
See \cite{MiSh:998} from 6.2 through end \S 6, and \cite{MiSh:1051} Theorems 3.11 and 6.3.})

\section{Existence and uniqueness} \label{s:uniqueness}

\begin{quotation}
\noindent 
\emph{Returning to our study of $\mcs(\cs, \xt_\cs) = \emptyset$ for an arbitrary $\cs$, we prove that for regular  
$\kappa \leq \xp_\cs$, there exists $\lambda$ such that $(\kappa, \lambda) \in \mcs(\cs)$, and moreover any such $\lambda$ is unique. 
A corollary will be that we can study cuts by looking in any nontrivial $\ma \in \ord(\cs)$ we like.}
\end{quotation}

\br
Remember from \S \ref{s:CSP} that for each $\ma \in \ord(\cs)$ we have a distinguished element $d_\ma \in X_\ma$ (not always the last element), 
and that $X_\ma$ is \textit{nontrivial} if $d_\ma$ is not a finite successor of the first element $0_\ma$. As each $X_\ma$ is discrete, 
the successor function $S$ is well defined; let $S^n$ abbreviate $n$-th successor and 
$S^{-k}$  abbreviate $k$-th predecessor.
Call $a \in X_\ma$ \emph{below the ceiling} if $a <_\ma d_\ma$ and moreover 
$a$ is not a finite predecessor of $d_\ma$, 
equivalently,\footnote{Equivalently because given $a$ below the ceiling, the set of $\{ x : a <_\ma x <_\ma d_\ma \}$ is bounded nonempty and 
definable so has a first element, the successor of $a$.}
if each of its finite successors is strictly less than $d_\ma$.

\begin{obs}[Existence, \cite{MiSh:998} 3.1] \label{o:existence}
If $\ma \in \ord(\cs)$ is nontrivial, then for any infinite regular $\kappa \leq \xp_\cs$ there is at least one\footnote{Notice that 
we are assuming nothing about the size of $\theta$, and so we write $\mcs(\cs)$, not $\mcs(\cs, \xt_\cs)$.} infinite regular 
$\theta$ such that $(\kappa, \theta) \in \mc(\cs)$, witnessed by a $(\kappa, \theta)$-cut in $X_\ma$. 
Similarly 
there is some $(\theta^\prime, \kappa)$-cut in $X_\ma$.
\end{obs}

\begin{proof}
By induction on $\alpha < \kappa$ we choose elements $a_\alpha \in X_\ma$ such that (i) $\beta < \alpha \implies a_\beta <_\ma a_\alpha$
and (ii) each $a_\alpha$ is below the ceiling. 
At $\alpha = 0$, let $a_0 = 0_\ma$, so ``below the ceiling'' holds as $X_\ma$ is nontrivial. 
At $\alpha = \beta +1$, let $a_\alpha = S(a_\beta)$. 
At $\alpha$ limit, by inductive hypothesis, $( \langle a_\beta : \beta < \alpha \rangle, \langle S^{-k}(d_\ma) : k < \omega \rangle)$ 
is a pre-cut and $\cf(\alpha) + \aleph_0 < \kappa + \aleph_0 \leq \xp_\cs$; recalling the definition of $\xp_\cs$, the pre-cut is filled. Let $a_\alpha$ be an  element filling it. Having built $\langle a_\alpha : \alpha < \kappa \rangle$, 
let $\theta$ be the coinitiality of the nonempty set $B = \{ b \in X_\ma : \alpha < \kappa \implies a_\alpha <_\ma b \leq d_\ma \}$. 
Note its coinitiality cannot be 1, since $\bar{a}$ is strictly increasing, so if $b$ is above $\bar{a}$ then so is $S^{-1}(b)$. So $\theta$ is infinite 
and regular. A parallel argument in the other direction gives $\theta^\prime$. 
\end{proof}

\begin{obs}[Treetops for definable sub-trees, \cite{MiSh:998} 2.14]
Given $\ma \in \ord(\cs)$ and suppose $\mct \subseteq \mct_\ma$ is a definable subtree. Let $\kappa < \xt_\cs$ be regular and suppose 
$\langle c_\alpha : \alpha < \kappa \rangle$ is a $\tlf_\ma$-strictly increasing sequence of elements of $\mct$. 
Then there is $c_* \in \mct$ which is an upper bound for the sequence. 

If in addition $\kappa < \xp_\cs$, we can also assume $n_* := \maxdom(c_*) \in X_\ma$ is below the ceiling $($in slight abuse of 
notation we may call $c_*$ a treetop below the ceiling$)$. 
\end{obs}

\begin{proof}
By definition of $\xt_\cs$, there is some treetop $c \in \mct_\ma$, but a priori it may not be in $\mct$. The set 
$\{ \lgn(c^\prime) : c^\prime \tlf_\ma c $ and $c^\prime \in \mct \}$ is a bounded nonempty definable subset of 
$X_\ma$ so has a maximum element $a$. Then $c_* = c \rstr a$ works. 

To ensure the second clause, when we first get the treetop $c \in \mct_\ma$, if it is not below the ceiling, then before proceeding 
notice that 
\[ ( \{ \maxdom(c_\alpha) : \alpha < \kappa \},  \{ S^{-k}(\maxdom(c)) : k < \omega \} ) \]
is a pre-cut, but not a cut, as $\kappa + \aleph_0 < \xp_\cs$. Let $n$ realize it, and replace $c$ by $c \rstr n$, then continue the argument. 
\end{proof}

Now for a main lemma. 

\begin{lemma}[Uniqueness, \cite{MiSh:998} 3.1]
Suppose $\kappa$ is regular, 
$\kappa < \min \{ \xp^+_\cs, \xt_\cs \}$. 
Then there is one and only one $\theta$ such that $(\kappa, \theta) \in \mcs(\cs)$.  
Moreover, $(\kappa, \theta) \in \mc(\cs)$ iff $(\theta, \kappa) \in \mcs(\cs)$. 
\end{lemma}

\begin{proof}
We have ``one'' from above, let's show ``only one.'' 
By transitivity of equality, it will suffice to show that if we are given $\ma, \mb \in \ord(\cs)$, 
$ ( { \langle a^0_\alpha : \alpha < \kappa \rangle}, {\langle b^0_\epsilon : \epsilon < \theta_0 \rangle})$ 
representing a $(\kappa, \theta_0)$-cut in $X_\ma$, and 
$ ({\langle b^1_\epsilon : \epsilon < \theta_1 \rangle}, {\langle a^1_\alpha : \alpha < \kappa \rangle})$ 
representing a $(\theta_1,\kappa)$-cut in $X_\mb$, then $\theta_0 = \theta_1$. 

Let $\bc \in \ord(\cs)$ be such that $X_\bc = X_\ma \times X_\mb$. So we can think of elements of $\mct_\bc$ 
as sequences of pairs with first coordinate in $X_\ma$, second coordinate in $X_\mb$. 
Let $\mct'$ be the definable subtree of $\mct_\bc$ consisting of all $x \in \mct_\bc$ strictly increasing 
in the first coordinate and strictly decreasing in the second coordinate, i.e. such that  
\[ m <_\bc n \leq_\bc \maxdom(x) \mbox{ implies } x(m,0) <_\ma x(n,0) \mbox{ and } x(m,1) <_\mb x(n,1). \]
Note that $c \in\mct'$ is a function whose domain is contained in $X_\bc$.

Now by induction on $\alpha < \kappa$ let's choose a path $c_\alpha$ (with $n_\alpha := \maxdom(c_\alpha)$) 
through the tree $\mct^\prime$ such that: 
\begin{itemize}
\item $\beta < \alpha \implies c_\beta \tlf c_\alpha$,
\item each $c_{\alpha} \in \mathcal{T}'$ is below the ceiling,
\item $c_\alpha(n_\alpha,0) = a^0_\alpha$ and $c_\alpha(n_\alpha, 1) = a^1_\alpha$.
\end{itemize}
With this path we are ``threading together the $\kappa$-sides of the cuts.''  Our strategy is similar to \S \ref{s:sym}. 
At $\alpha = 0$, $c_0 = \langle a^0_0, a^1_0 \rangle$, $n_0 = 0_\bc$. 
At $\alpha = \beta+1$, $c_\beta$ is below the ceiling, so we can concatenate: $c_\alpha = c_\beta ~^\smallfrown \langle a^0_\alpha, a^1_\alpha \rangle$ 
and $n_\alpha = n_\beta +1$. 
At $\alpha < \kappa$ limit, $\cf(\alpha) < \min \{ \xp_\cs, \xt_\cs \}$ so we can choose a treetop $c_*$ for $\langle c_\beta : \beta < \alpha \rangle$ 
which belongs to $\mct^\prime$ and is below the ceiling. If necessary back up: let $n_* = \maxdom(c_*)$, so 
\[ \{ n \leq_\bc n_* : c_*(n,0) <_\ma a^0_\alpha \mbox{ and } a^1_\alpha <_\mb c_*(n, 1) \} \]
is nonempty, bounded, definable and has a greatest element $n_{**}$ (necessarily below the ceiling). 
Let $c_\alpha = c_* \rstr_{n_{**}} ~^\smallfrown \langle a^0_\alpha, a^1_\alpha \rangle$ (we may concatenate as we are below the ceiling), 
so $n_\alpha = \maxdom(c_\alpha)$ is still below the ceiling. 

In this way we define $\langle c_\alpha : \alpha < \kappa \rangle$. As $\kappa < \xt_\cs$, there is a treetop 
(which does not need to be below the ceiling) 
$c_\star \in \mct^\prime$. Let $n_\star = \maxdom(c_\star)$. 
`Stitching together' the sequences $\langle a^{\delta}_{\alpha} : \alpha < \kappa \rangle$ for $\delta = 0,1$ has given us a strictly increasing sequence $\langle n_{\alpha} : \alpha < \kappa \rangle$ in $X_{\mathbf{c}}$, the domain of $c_\star$;  
we now look for two decreasing sequences  in $X_{\mathbf{c}}$ corresponding to the $\langle b^{\delta}_{\beta} : \beta < \theta_{\delta} \rangle$ for $\delta = 0,1$ which each form a cut with $\langle n_{\alpha} : \alpha < \kappa \rangle$. 
For $\beta < \theta_0$ and $\gamma < \theta_1$, we can define:$$n_{\beta}^0 = \max \{ n \leq_{\mathbf{c}} n_\star : c_\star(n,0) <_{\mathbf{a}} b^0_{\beta} \},$$
$$n_{\gamma}^1 = \max \{ n \leq_{\mathbf{c}} n_\star : c_\star(n,1) >_{\mathbf{b}} b^1_{\gamma} \}.$$
Now $(\langle n_{\alpha} : \alpha < \kappa \rangle, \langle n^0_{\beta} : \beta < \theta_0\rangle )$ is a pre-cut in $X_{\mathbf{c}}$.  
If it were realized by some $n \in X_{\mathbf{c}}$, then by construction and by the monotonicity built into elements of $\mct^\prime$, 
$c_\star(n,0)$ would realize the cut $( { \langle a^0_\alpha : \alpha < \kappa \rangle}, {\langle b^0_\epsilon : \epsilon < \theta_0 \rangle})$.
Likewise, ${(\langle n_{\alpha} : \alpha < \kappa \rangle, \langle n^0_{\beta} : \beta < \theta_0 \rangle )}$ must be a cut in $X_{\mathbf{c}}$; if 
it were realized by some $n$, $c_\star(n,1)$ would realize 
$ ({\langle b^1_\epsilon : \epsilon < \theta_1 \rangle}, {\langle a^1_\alpha : \alpha < \kappa \rangle})$. (Since $n \leq n_\star$, 
automatically $n \in \dom(c_\star)$, and now remember the range of $c_\star$ is monotonic in each coordinate.)

So as $\theta_0, \theta_1$ are regular cardinals, we conclude that $\theta_0 = \theta_1$. 
\end{proof}

This proof shows something quite strong:\footnote{In the notation of the proof, the graph of the definable map is given by the range 
of $c_\star$. See \cite{MiSh:998} from Definition 3.2 to the end of \S 3. 
For related open problems, see the introduction to \cite{MiSh:1070}.} 
for $\kappa < \min \{ \xp^+_\cs, \xt_\cs \}$, 
definable monotonic maps exist between \emph{any} two strictly monotonic $\kappa$-indexed sequences in any two of our orders. 
A fortiori for such $\kappa$ (recall \ref{o:existence}), $(\kappa, \theta) \in \mcs(\cs)$ iff 
$(\theta, \kappa) \in \mcs(\cs)$. 

\begin{concl}
Suppose $\kappa < \min(\mathfrak{p}^+_{\mathbf{s}}, \mathfrak{t}_{\mathbf{s}})$.  In order to show that $(\kappa, \theta) \not\in \mathscr{C}(\mathbf{s})$, it suffices to prove that \emph{some} strictly increasing $\kappa$-indexed sequence in \emph{some} $X_\ma$ $($where $\ma \in \ord(\cs)$ 
is nontrivial$)$ has coinitiality in $X_\ma$ not equal to $\theta$.  
\end{concl}

\noindent So for $\kappa < \min \{ \xp^+_\cs, \xt_\cs \}$, the function 
$\lcf(\kappa)$ giving this coinitiality is well defined.\footnote{The name ``lower cofinality'' comes from a related property of ultrafilters: 
$\lcf(\omega, \de)$ means the co-initiality of the set above the diagonal embedding of $\omega$ in $(\omega, <)^I/\de$.} 

\section{Local saturation}

\begin{quotation}
\noindent 
\emph{We prove that every CSP has a certain basic amount of saturation: any element in 
any $X_\ma$ which may be described by ${\kappa < \min \{ \xp_\cs, \xt_\cs \}}$ definable sets in $N^+$ must exist.}
\end{quotation}

\br

In our next constructions, it will be useful to appeal to ambient saturation. 
By an ``\emph{$\ord$-type}'' we mean a type $p(x_0, ..., x_{n-1})$ with parameters in $N^+$ such that for each $i < n$, there is some $\ma_i \in \ord(\cs)$ so that $p(\overline{x})$ implies $x_i \in X_{\ma_i}$.
Let's prove that every $\ord$-type over a small set is realized. Since we have closure under Cartesian products, 
without loss of generality the $x_i$ are all in the same $X_{\mathbf{a}}$. This is \cite{MiSh:998} 4.1. 

\begin{lemma} \label{satlem}
If $\kappa < \min \{ \mathfrak{p}_{\mathbf{s}}, \mathfrak{t}_{\mathfrak{s}}\}$, then every $\ord$-type over a set of size $\kappa$ is realized.
\end{lemma}

\begin{proof}
We prove this by induction on infinite $\kappa$, so suppose that $\kappa = \aleph_0$ or the lemma holds for all infinite $\theta < \kappa$. 
Let $p(\overline{x}) = \{ \varphi_i(\overline{x}, a_i) : i < \kappa \}$ be an $\ord$-type, so finitely satisfiable in some fixed $X_{\mathbf{a}}$.  
We construct a path $\langle c_{\alpha} : \alpha \leq \kappa \rangle$ through $\mathcal{T}_{\mathbf{a}}$ so that $c_{\kappa}$ will provide us with the realization of $p(\overline{x})$.  By a second (``internal'') induction on $\alpha < \kappa$,  we choose $c_{\alpha} \in \mathcal{T}_{\mathbf{a}}$ and let $n_{\alpha}:= \maxdom(c_\alpha) \in X_{\mathbf{a}}$ to 
satisfy:
\begin{enumerate} [i)]
\item $\beta < \alpha \implies c_\beta \tlf c_\alpha$,
\item $n_{\alpha}$ is below the ceiling,
\item $i < \beta \leq \alpha$ and $n_{\beta} \leq_{\mathbf{a}} n \leq_{\mathbf{a}} n_{\alpha}$ imply $N_* \models \varphi_i[c_{\alpha}(n), a_i]$.
\end{enumerate}
For $\alpha = 0$, this is trivial.  For $\alpha = \beta +1$, 
$\{ \varphi_i(\overline{x}, a_i) : i \leq \beta \}$ is a (partial) type of strictly smaller cardinality, so we can use 
the external inductive hypothesis (or the definition of type if finite) to find $\overline{d} \in X_{\mathbf{a}}$ realizing it. 
 Let $c_{\alpha} = c_{\beta} \smallfrown \langle \overline{d} \rangle$. 

For limit $\alpha \leq \kappa < \min \{ \mathfrak{p}_{\mathbf{s}}, \mathfrak{t}_{\mathfrak{s}}\}$,  by Observation 10.1 there is a treetop 
$c_* \in \mct_\ma$ for $\langle c_{\beta} : \beta < \alpha \rangle$ with $n_* = \maxdom(c_*)$ below the ceiling.  
Now we correct to preserve condition (iii): for each $i<\alpha$, define 
$$n(i) = \mathrm{max}\{ n \leq_{\mathbf{a}} n_* : N_* \vDash \varphi_i[c_*(m), a_i] \mbox{ for all } n_{i+1} \leq_{\mathbf{a}} m \leq_{\mathbf{a}} n \}.$$
Then $\bar{n} := \langle n_\beta : \beta < \alpha \rangle$ is an increasing sequence, and $\{ n(i) : i < \alpha \}$ is a set all of whose elements are above $\bar{n}$. Let $\gamma$ be the co-initiality of this set.  Since $\cf(\alpha) + \gamma < \xp_\cs$, there is $n_{**}$ 
realizing the pre-cut so described. 
Let $c_{\alpha} = c_* \restriction_{n_{**}}$.  Note this case includes $\alpha = \kappa$, and $c_\kappa(n_\kappa)$ will realize the type. 
\end{proof}

\section{Upgraded trees} \label{s:upgraded}

\begin{quotation}
\noindent 
\emph{Any CSP has a certain basic amount of arithmetic, which allows us to build more powerful trees. Since this is 
easily true in our two running examples, we omit the details here.}
\end{quotation}

\br

So far, all of the trees we've considered were fairly simple: if $c \in \mct$ were evaluated at some element in its domain, we would get 
an element of a linear order, or a tuple of elements. In \cite[\S 5]{MiSh:998} it's shown that in any CSP we may do basic G\"odel coding.\footnote{Moreover, 
it's fairly natural. For example, to define addition on a given order $X_\ma$, let
$\vp_+(x,y,z) = (\exists \eta \in \mct_\ma : ( \lgn(\eta) = y ~\land ~ \eta(0) = x ~\land ~\eta(y-1) = z ~ \land
~(\forall i)(i<\lgn(\eta)) \rightarrow \eta(S(i)) = S(\eta(i))~) \}$. For multiplication, modify the previous formula to $\vp_\times$ which
increments by $x$ instead of by $1$. For exponentiation, increment by a factor of $x$, i.e.
$\eta(S(i)) = \eta(i) \cdot x$. And so on.}  This allows us to 
code subsets of certain orders as elements of other orders, and to find pairs of orders $\ma, \mb \in \ord(\cs)$ so that 
$\mct_\ma$ may be definably identified with a definable subset of $X_\mb$ -- in this case, we say ``$\ma$ is coverable by $\mb$.''  
We can also find $\ma$ and $\mb$ such that the tree $\mct_{\ma \times \ma}$ can be identified with a definable subset of $X_\mb$ -- 
in this case, we say ``$\ma$ is coverable as a pair by $\mb$.''\footnote{This is where we use the condition in \S \ref{s:CSP}  
about a well behaved ordering on some pair.} 

Since the coding and decoding are definable (in $M^+_1$), the effect of this is to free us to build more powerful trees. 
If $\ma$ is covered by $\mb$, we can find a definable subtree of $\mct_\mb$ whose elements are, say, 
tuples of elements of $\mct_\ma$ -- in other words, tuples of functions from $X_\ma$ to itself.  Or perhaps they are tuples 
consisting of some element of $X_\ma$, some function from $X_\ma$ to itself, some subset of $X_\ma$, and some other element of
$X_\ma$, all relating to each other in a certain way.

\begin{concl}
In what follows we use such ``upgraded trees'' without further comment, and given some $X_\ma$, we freely use the partial functions
$+$, $\times$, and exp.
\end{concl}

\section{Towards determining $\mathscr{C}(\mathbf{s})$} \label{s:aleph-0}

\begin{quotation}
\noindent 
\emph{With tools from the previous few sections, we work towards the central goal of showing that $\mcs(\cs, \xt_\cs) = \emptyset$. 
This includes a warm-up for the main lemma in the next section.}
\end{quotation}

\br

Let's first record that the analogue of Lemma \ref{symlem} holds for a general CSP, similarly to the ultrapower case. 
For details, see \cite[Lemma 6.1]{MiSh:998}. 

\begin{lemma} \label{symlem-CSP} 
Let $\cs$ be a CSP. 
If $\lambda \leq \mathfrak{p}_{\mathbf{s}}$ and $\lambda < \mathfrak{t}_{\mathbf{s}}$, then $(\lambda, \lambda) \not\in \mathscr{C}(\mathbf{s})$.
\end{lemma}

In order to prove that $\mathfrak{p}_{\mathbf{s}} \geq \mathfrak{t}_{\mathbf{s}}$, we will need to show that there are no \textit{asymmetric cuts} below $\mathfrak{t}_{\mathbf{s}}$.  To this end, we begin with an easier lemma as a warm-up for the main goal\footnote{The theorem in the next 
section will supercede ths lemma, but the present proof is simpler and motivates many of the ideas there.} in the next section. 
This is a condensed version of \cite{MiSh:998} \S 7. 

\begin{lemma} \label{firstlemma}
Suppose $(\aleph_0, \aleph_1) \notin \mcs(\cs, \xt_\cs)$. Then also $(\aleph_0, \lambda) \notin \mcs(\cs, \xt_\cs)$ for all 
regular $\lambda< \min \{ \xp^+_\cs, \xt_\cs \}$. 
\end{lemma}

\begin{proof}
We may reduce to proving:  if $\lambda$ is regular, 
$\aleph_1 < \lambda = \mathfrak{p}_{\mathbf{s}}$ and $\lambda < \mathfrak{t}_{\mathbf{s}}$, then $(\aleph_0, \lambda) \not\in \mathscr{C}(\mathbf{s})$.    Note this means  $(\aleph_0, \aleph_0) \notin \mcs(\cs)$ by \ref{symlem-CSP}, $(\aleph_1, \aleph_1) \notin \mcs(\cs)$ by 
$\ref{symlem-CSP}$, $(\aleph_0, \aleph_1) \notin \mcs(\cs)$ by assumption. So both $\aleph_0$ and $\aleph_1$ have co-initiality $\geq \aleph_2$. 

Let's assume that $(\aleph_0, \lambda) \in \mcs(\cs)$ and we will arrive at a contradiction. 

Once again we'll build a tree, but let's be a bit more careful in our setup.  First, choose some nontrivial $\ma$ and $\mb$ so that 
$\ma$ is coverable by $\mb$ in the sense of the previous section, \emph{i.e.} $\mct_\ma$ may be identified with a definable 
subset of $X_\mb$.  Let's choose to study\footnote{Why may we simply choose \emph{some} sequence 
in \emph{some} given $X_\ma$? Recall the end of \S \ref{s:uniqueness}.}
 the cut above the ``standard copy of $\omega$'' in $X_\ma$, i.e. the sequence given by $d_\delta = 0_\ma$ and 
$d_\delta = S^{\delta}(0_\ma)$, for $\delta < \omega$.  
By existence and uniqueness, there is a sequence $\langle e_\alpha : \alpha < \lambda \rangle$ in $X_\ma$ such that 
\[ ( \langle d_\delta : \delta < \omega \rangle, \langle e_\alpha : \alpha < \lambda \rangle) \]
form a cut in $X_\ma$. Without loss of generality, we may assume any two consecutive elements of the sequence $\bar{e}$ 
are infinitely far apart (if not, as $\lambda$ is regular uncountable, just thin the sequence by taking every element 
whose index is divisible by $\omega$). 

Second, fix $\aleph_1$-many distinguished elements of $X_\ma$: $\{ a_{i} \in X_{\mathbf{a}} : i < \aleph_1 \}$. \footnote{Note our 
setup ensures $X_\ma$ has uncountably many elements, and moreover, that any two elements of $X_\ma$ which are 
infinitely far apart have uncountably many elements between them.}

Third, let's fix our tree: let $X_{\mathbf{c}} = X_{\mathbf{a}} \times X_{\mathbf{a}} \times X_{\mathbf{b}}$, so that $\mct_{\mathbf{c}}$ consists of sequences of triples. Consider the definable subtree $\mct \subseteq \mct_\bc$ consisting of those $x$ such that for $n \leq \maxdom(x)$:
\begin{enumerate}[1)]
\item $x(n,0) <_\ma x(n,1)$,
\item $x(n,2)$ is a partial injective map $($\emph{i.e. an element of $\mct_\ma$ considered as a 
definable subset of $X_\mb$}$)$ from $X_\ma$ into the interval $[x(n,0), x(n,1)]_\ma := \{ a \in X_\ma : x(n,0) \leq_\ma a \leq_\ma x(n,1) \}$,
\item $n <_{\bc} m $ implies $x(m,0) <_\bc x(n,1)$ and $\dom( x(m,2) ) \subseteq \dom( x(n,2) )$. 
\end{enumerate}

Informally, for each $n$, the first two coordinates are endpoints of an interval, moving ``left'' towards $0_\ma$ as the path advances. The 
third is a partial function into that interval. Note this is allowed by the previous section: the third coordinate is an element of 
the tree $\mct_\ma$, so a partial function from $X_\ma$ to itself. 

By induction on $\alpha < \lambda$, we will\footnote{Our idea will be to choose a path through the tree so that for $\alpha < \lambda$,   
$c_\alpha(n_\alpha, 0)$ and $c_\alpha(n_\alpha, 1)$ align with $e_{\alpha+1}, e_\alpha$, and the function 
$c_\alpha(n_\alpha, 2)$ is an injective map whose range is bounded by these endpoints and whose domain contains the $\aleph_1$ distinguished constants.   
Advancing along a branch, we are effectively moving left through the nonstandard elements towards the standard ones.  
At each step, we injectively map at least $\aleph_1$-many distinct elements into the interval at hand. 
If we can ``carry'' this all the way, we'll find a contradiction on grounds of size as many constants overspill in to the domain(s) of some third coordinate function(s) with standard endpoints.}
 build a path $\langle c_{\alpha} \in \mct : \alpha < \lambda \rangle$ with $n_{\alpha} = \textrm{max dom}(c_{\alpha})$ such that:
\begin{enumerate}[i)]
\item $n_{\alpha}$ is below the ceiling,
\item $e_{\alpha+1} \leq_{\mathbf{a}} c_{\alpha}(n_{\alpha},0) <_{\mathbf{a}} c_{\alpha}(n_{\alpha},1) \leq_{\mathbf{a}} e_{\alpha}$,
\item $a_i \in \mathrm{dom}(c_{\alpha}(n, 2))$ for each $i < \aleph_1$ and $n \leq_{\mathbf{a}} n_{\alpha}$.
\end{enumerate}

\noindent \underline{For $\alpha = 0$}.  We would like a definable partial injection\footnote{Note this is 
an $\ord$-type:  we are asking  
for an element of $\mct_\ma$ (definably identified with a definable subset of $X_\mb$); we'll use ``definable partial injection'' 
in this way for the rest of the proof.  The type uses parameters $\{ a_i : i < \aleph_1 \} \cup \{ e_1, e_0 \}$ plus finitely many needed to code $\mct_\ma$ in $X_\mb$.} 
 from $X_\ma$ to itself whose domain contains $\{a_i : i < \aleph_1 \}$, with range in the interval $[e_1,  e_0]_\ma$.  Since there are uncountably many elements between $e_1$ and $e_0$, this is consistent, and may be described by a type over a set of size $\aleph_1 < \min(\mathfrak{p}_{\mathbf{s}}, \mathfrak{t}_{\mathbf{s}})$.  So by Lemma \ref{satlem}, this is realized by some $f \in X_{\mathbf{b}}$.  Let $c_0 = (e_1, e_0, f)$ with $n_0 = 0_{\mathbf{c}}$.

\sbr
\noindent \underline{For $\alpha = \beta+1$}. As before, the type describing a definable partial injection from $X_\ma$ to $X_\ma$ whose domain contains all $\{a_i : i < \aleph_1 \}$ and is contained in the definable set $\mathrm{dom}(c_{\beta}(n_{\beta}, 2))$, and whose range is contained in $(e_{\alpha+1}, e_\alpha)_\ma$, is a 
consistent type over a set of size $\aleph_1$ so realized by some $g$. Since $n_{\beta}$ is below the ceiling, we may concatenate $(e_{\alpha+1}, e_\alpha, g)$ to $c_{\beta}$.

\sbr
\noindent \underline{For $\alpha$ a limit}.  Since $\alpha < \lambda < \mathfrak{t}_{\mathbf{s}}$, there is some upper bound $c_* \in \mct$ for $\langle c_{\beta} : \beta < \alpha \rangle$ with $n_*$ below the ceiling.  To ensure we did not overshoot $e_{\alpha}$, we can restrict $c_*$ to $N = \mathrm{max} \{ n \leq_{\mathbf{c}} n_* : e_{\alpha} < _{\mathbf{a}} c_*(n,0) \}$.  We also need to ensure that all the $a_i$ remain in the domain 
of all the third coordinate functions.  So for each $i < \aleph_1$, we can define $n(i) = \mathrm{max}\{ n \leq_{\mathbf{c}} N : a_i \in \mathrm{dom}(c_*(n,2)) \}$.  By inductive 
hypothesis, $n_\beta <_\bc n(i) $ for each $i < \aleph_1$ and each $\beta < \alpha$. 
The pre-cut\footnote{Set notation as the right need not be a descending sequence, but will have co-initiality $ \leq \aleph_1$.}
 $(\{ n_{\beta} : \beta < \alpha \}, \{ n(i) : i < \aleph_1 \})$ is filled by some $n_{**}$ since $\cf(\alpha)+ \aleph_1 < \mathfrak{p}_{\mathbf{s}}$.  
The Or-type describing a definable partial injection $f$ from $X_\ma$ to $X_\ma$ with $\{a_i : i < \aleph_1 \} \subseteq \mathrm{dom}(f) \subseteq \mathrm{dom}(c_*(n_{**}, 2))$ is again consistent and requires $\aleph_1$ parameters. Let $f_\alpha$ realize it.  Finally, let $c_{\alpha} = {c_* \restriction_{n_{**}}} ^\smallfrown (e_{\alpha+1}, e_{\alpha}, f_{\alpha})$.  This completes the induction. 

\sbr
As $\lambda < \xt_\cs$, our sequence $\langle c_\alpha : \alpha < \lambda \rangle$ has an upper bound $c_\star \in \mct$. Let 
$n_\star = \maxdom(c_\star)$. 
For each $i < \aleph_1$, define $m(i) = \max \{m \leq_{\mathbf{a}} n_{\star} : a_i \in \mathrm{dom}(c_{\star}(m,2)) \}$,  
point where $a_i$ ``fell out'' of the domain of $(c_\star(m,2))$. 
By our inductive construction, for each $i<\aleph_0$ and each $\alpha < \lambda$, $n_\alpha <_\bc m(i)$. 
So for each $i < \aleph_0$, it must be that $c_\star(m(i), 1) = d_\delta$ for some $\delta$ (recalling that $(\bar{d}, \bar{e})$ is a cut). 
So by the pigeonhole principle, there is some standard (=finite) $d_\star$ such that $c_\star(m(i), 1) = d_\star$ for uncountably many $i$.
Let $W$ be the set of such $i$.  

Now for the contradiction. 
Notice what we've shown is that for each $i \in W$, there is $m(i)$ such that the function $f_i := c_\star(m(i), 2)$ is into 
$[0, d_\star]_\ma$ and $f_i$ has $a_i$ in its domain. A priori, this does {not} say that there is a single $f_i$ with 
many elements in its domain, so we now have two cases. If there exists an infinite $W^\prime \subseteq W$ on which 
the function $i \mapsto m(i)$ is constant, let $m_\star$ denote ths constant value. Then $f_{m_\star} = c_\star(m_\star, 2)$ 
is indeed an injection with infinite domain and finite range, a contradiction. Otherwise, there is an infinite set $W^{\prime\prime}$ on 
which the function $i \mapsto m_i$ is one-to-one. Then the function $g$ given by $a_i \mapsto c_\star(m_i, 2)(a_i)$ is one-to-one on $W^{\prime\prime}$ 
because $k <_\bc \ell$ implies $c_\star(\ell, 1) <_\ma c_\star(k,0)$. So the restriction of $g$ to any subset of $W^{\prime\prime}$ of size $d_\star+2$  
is a definable injection of a finite set of size $d_\star+2$ into a set of size $d_\star$, a contradiction. This completes the proof.  
\end{proof}

This section isn't only pedagogical, but reproduces the arc of writing \cite{MiSh:998}. 
Although now superceded by the main result of the next 
section, this was a first major advance in understanding the picture of cuts. 
At the time the case of $(\lambda, \lambda^+)$-cuts with $\lambda^+ = \xp$ remained mysterious.

 \section{No asymmetric cuts}

\begin{quotation}
\noindent 
\emph{Extending ideas from the previous section, we prove the main lemma ruling out asymmetric cuts below treetops, which completes the proof that $\mcs(\cs, \xt_\cs) = \emptyset$.}
\end{quotation}

\br

We'll now sketch the proof of the lemma ruling out all asymmetric cuts below treetops, \cite{MiSh:998} \S 8. 
For this will need to upgrade the argument from \S \ref{s:aleph-0} to handle a $(\kappa, \lambda)$-cut for arbitrary
$\kappa < \lambda = \xp_\cs < \xt_\cs$. Some points to notice in the earlier proof: 
\begin{enumerate}
\item There, the presumed cut had its left side $\bar{d}$ consisting of standard elements, and its right side $\bar{e}$ infinitely spaced.
\item We ``carried'' a set of size $\aleph_1$ into the left side where we got a contradiction for size reasons. (The choice of $\bar{d}$ 
mattered: we would have had trouble 
with the contradiction if $\bar{d}$ were some widely spaced $\omega$-indexed sequence.)
\item The presumed cut was of type $(\aleph_0, \lambda)$, and $\aleph_1 < \min \{ \xp_\cs, \xt_\cs \}$. We needed this inequality to 
apply local saturation at steps $\alpha < \lambda$. Notice $\aleph_1 = \aleph_0^+$. 
\end{enumerate}
When it comes to more general $(\kappa, \lambda)$-cuts: 
\begin{enumerate}
\item[(1)$^\prime$] The point above was that in the sequence $(\bar{d}, \bar{e})$ witnessing the cut, successive elements of the left hand side 
`grow in cardinality,' whereas successive elements of the right-hand side are `widely spaced'.
\item[(2)$^\prime$] If $\bar{d}$ is a $\kappa$-indexed sequence, many of its elements may be far apart. It won't in general be sufficient to carry a set of size $\kappa^+$ into the sequence $\bar{d}$ to obtain a contradiction.\footnote{Think of the diagonal embedding of an uncountable $\kappa$ in a regular ultrapower.} We need to keep better track of size. However, we'll see that CSPs have a natural internal notion of cardinality. 
\item[(3)$^\prime$] If $\kappa^+ = \lambda = \xp_\cs$, for a contradiction we would need to carry $\kappa^+$ constants as we go along, 
but with $\kappa^+ = \xp_\cs$ we can't obviously apply local saturation which requires $<$. We may solve this by growing the number of 
constants we carry with $\alpha$: the constant $a_\alpha$ is added to the domain of the functions by stage $\alpha+1$. So at each stage 
$\alpha < \kappa^+$, we have $\leq \kappa$ constants to carry. (If $\kappa^+ < \lambda$, this is excessive caution, but if 
$\kappa^+ = \lambda = \xp_\cs$ it is key.) 
\end{enumerate}

\br

\noindent 
Let us now set the stage. 
First, we'll need an internal notion of cardinality, following \cite{MiSh:998} \S 5.  Suppose $\ma, \mb \in \ord(\cs)$ and let 
$\bc = \ma \times \mb$. Let $\prt(\ma, \mb)$ be the definable subtree of $\mct_\bc$ consisting of $x$ such that 
$\{ ~( x(n,0), x(n,1) ) : n \leq_\bc \maxdom(x) ~\}$ is the graph of a partial one-to-one map from $X_\ma$ to $X_\mb$.  
Working in $M^+_1$, if $A \subseteq X_\ma$ and $B \subseteq X_\mb$ are definable sets, let us write ``$|A| \leq^\cs |B|$'' to mean there exists 
$x \in \prt(\ma, \mb)$ such that $A \subseteq \{ x(n,0) : n \leq_\bc \maxdom(x) \}$ and 
$B \subseteq \{ x(n,1) : n \leq_\bc \maxdom(x) \}$.  Write ``$|A| <^\cs |B|$'' if   
$|A| \leq^\cs |B|$ and and no $x \in \prt(\ma, \mb)$  
witnesses $|B| \leq^\cs |A|$. 
This definition allows us to make sense of relative size for any elements $a, b$ in some $X_\ma$: let 
``$|a| \leq^\cs |b|$'' mean $| \{ x \in X_\ma : x \leq_\ma a \}| \leq^\cs | \{ x \in X_\ma : x \leq_\ma b \}|$. 

Second, we'll need to select a suitable cut. Suppose $\kappa < \lambda = \xp_\cs < \xt_\cs$ and 
$(\kappa, \lambda) \in \mcs(\cs)$. Suppose we are given a nontrivial $\ma$. 
Then with a little work we may choose a cut $(\langle d_\beta : \beta < \kappa \rangle, \langle e_\alpha : \alpha < \lambda \rangle)$ in $X_\ma$ so that 
the left-hand side grows in internal cardinality, meaning that $\beta < \beta^\prime$ implies $|d_\beta| <^\cs |d_{\beta^\prime}|$ in the 
sense just given, and the right-hand side is 
widely
 spaced, meaning that $e_{\alpha+1} +_\ma e_{\alpha+1} <_\ma e_\alpha$.\footnote{This construction is \cite{MiSh:998}, Claim 8.2. It amounts to building a branch through a 
carefully designed tree, noting that branches are long and it is easy to satisfy these conditions when the values are finite. 
Recall from \S \ref{s:upgraded} that we have addition within each $X_\ma$. } 

Third, we'll need a fact:\footnote{Any symmetric function such that $\alpha < \alpha^\prime < \beta$ implies 
$g(\alpha, \beta) \neq g(\alpha^\prime, \beta)$ will do.}

\begin{fact} \label{fact5}
 There is some symmetric $g: \kappa^+ \times \kappa^+ \rightarrow \kappa$ such that for any $W \subset \kappa^+$ with $|W| = \kappa^+$, then $\mathrm{sup}(\mathrm{range}(g \restriction_{W \times W})) = \kappa$.
 \end{fact}

We're ready for the main lemma; we'll sketch here the main points of the proof. 

\begin{theorem-kl}[\cite{MiSh:998} Theorem 8.1] Let $\cs$ be a cofinality spectrum problem and let $\kappa, \lambda$ be regular. 
If $\kappa < \lambda \leq \mathfrak{p}_{\mathbf{s}}$ and $\lambda <  \mathfrak{t}_{\mathbf{s}}$, 
then $(\kappa, \lambda) \not\in \mcs(\cs, \xt_\cs)$.
\end{theorem-kl}

\begin{proof}[Proof sketch] 
Suppose for a contradiction that $(\kappa, \lambda) \in \mcs(\cs, \xt_\cs)$. This time, choose some nontrivial $\ma$ and $\ma^\prime \in \ord(\cs)$ 
so that $\ma$ is coverable as a pair by $\ma^\prime$, i.e., $\mct_{\ma\times\ma}$ may be definably identified with a definable subset of 
$X_{\ma^\prime}$. Fix for a contradiction a cut $(\langle d_\beta : \beta < \kappa \rangle, \langle e_\alpha : \alpha < \lambda \rangle)$ in $X_\ma$ 
as described above, with the left side strictly growing in internal cardinality, and the right side well spaced. 
Fix $g: \kappa^+ \times \kappa^+ \rightarrow \kappa$, an outside function satisfying Fact \ref{fact5} 
which will help in our bookkeeping. 

Finally, fix an order $\mb$ such that $X_\mb = X_\ma \times X_\ma \times X_\ma \times X_{\ma^\prime} \times X_{\ma^\prime} \times X_{\ma^\prime}$. 
This $\mb$ is the order we'll work in, and our chosen tree will be a definable subtree of $\mct_\mb$. \emph{In this proof, cardinality will 
always mean internal cardinality.}

\sbr

Let $\mct$ be the definable subtree of $\mathcal{T}_{\mb}$ consisting of $x$ as follows. (The informal small print describes 
what the intention or use will be in the inductive construction.)

\begin{itemize}
\item $n <_\mb n^\prime \leq_\mb \maxdom(x)$ implies 
\[ x(n^\prime, 0) <_\ma x(n^\prime, 1) <_\ma x(n^\prime, 2) <_\ma x(n,1). \]

\begin{tiny}
\noindent The first three coordinates move leftwards together towards the cut. First is a `marker,' 
keeping track of leftward progress, followed by the endpoints of an interval as before.
\end{tiny}
\item $x(n,3)$ is a nonempty subset of $X_\ma$ of size ``no more than half of $X_\ma$,'' i.e. 
\[ |x(n,3) | \leq^\cs |X_\ma \setminus x(n,3) |. \]

\begin{tiny}
\noindent This is the definable domain of the definable function $x(n,5)$. We could incorporate this into the definition of 
$x(n,5)$ but we list it separately for clarity. The size constraint will help in the induction. 
\end{tiny}

\item $x(n,4)$ is a symmetric two-place function with domain $x(n,3) \times x(n,3)$ and range $\subseteq X_\ma$.

\begin{tiny}
\noindent This will be our ``distance estimate function,'' which takes in a pair of elements in the domain of $x(n,5)$ and 
returns a lower bound on how far apart their images under $x(n,5)$ must be, see next. 
\end{tiny}
\item $x(n,5)$ is a 1-to-1 function from $x(n,3)$ into the interval $( x(n,1), x(n,2) )_\ma$ which 
\emph{respects the distance estimate function}, meaning\footnote{Here an expression like $|c-d|$ stands for  
$| \{ z  \in X_\ma : c \leq z \leq d \} |$.} 
\[ a \neq b \in x(n,3) \implies |x(n,4) (a,b) | \leq^\cs | x(n,5)(a) - x(n,5)(b) |. \] 

\begin{tiny}
\noindent Note that the function $x(n,5)$ forces the 
interval it maps into to be large in a sense controlled by the distance estimate function.   
\end{tiny}

\item if $n <_\mb n^\prime \leq_\mb \maxdom(x)$, then for any $a, b \in X_\ma$ such that 
\[ (\forall m) (n \leq_\mb m \leq_\mb n^\prime \implies \{ a,b \} \subseteq x(m,3) \]
we have that $x(n,4)(a,b) = x(n^\prime, 4)(a,b)$. 

\begin{tiny}
\noindent As long as two elements stay continuously in the domain of the fifth-coordinate function, 
the distance estimate function on them does not change.  This will be crucial to handling overspill. 
\end{tiny}

\end{itemize}

\noindent Now by induction on $\alpha < \lambda$ we'll choose $c_\alpha \in \mct$ and $n_\alpha = \maxdom(c_\alpha)$, with 
$\beta < \alpha \implies c_\beta \tlf c_\alpha$, satisfying the hypotheses below. 
When $\alpha = \beta+1$ is a successor ordinal $<\kappa^+$, we'll also choose a new constant $y_{\beta+1}$. (In this proof we'll 
have $\kappa^+$ distinguished elements of $X_\ma$ comprising the set we ``carry along towards the cut,'' and there is no harm in assuming they are all indexed by successor ordinals $<\kappa^+$.) 
In the induction, we would like to ensure:

\sbr

\noindent\underline{For all $\alpha < \lambda$}.

\begin{enumerate}[(a)]
\item if $\beta < \alpha$, then $e_{\alpha+1} \leq_\ma c_\alpha(n_\alpha, 0) <_\ma c_\alpha(n_\alpha, 1) <_\ma c_\alpha(n_\alpha, 2) <_\ma e_{\beta+1}$, and if $\alpha = \beta+1$, then in addition $c_\alpha(n_\alpha, 0) = e_{\alpha+1}$. \\
\begin{tiny}
\noindent The marker moves left followed by the two endpoints, and keeps pace with $\bar{e}$.
\end{tiny}

\item For all $\gamma < \min \{ \alpha, \kappa^+ \}$, 
\begin{itemize}
\item $y_{\gamma+1} \in c_\alpha(n_\alpha,3)$  
\\ \begin{tiny}
\noindent All constants of small index are in the domain.
\end{tiny}
\item $(\forall m)(n_{\gamma+1} \leq_\mb m \leq_\mb n_\alpha \implies y_{\gamma+1} \in c_\alpha(m,3) )$
\\ \begin{tiny}
\noindent And have stayed there ever since they were put in.
\end{tiny}
\item for all $\zeta+1 < \gamma+1$ and all $m$ such that $n_{\gamma+1} \leq_\mb m \leq_\ma n_\alpha$, 
\[ c_\alpha (m,4) (y_{\zeta+1}, y_{\gamma+1}) = d_{g(\zeta+1, \gamma+1)} \]
recalling $g$ is the external bookkeeping function fixed above. 
\\ \begin{tiny}
\noindent The distance estimate function assigns this pair 
to be at least as far apart as $| [0, d_{g(\zeta+1, \gamma+1)}]_\ma|$. Recall the distance estimate will not change 
as long as they both stay in the domain.
\end{tiny}
\end{itemize}
\end{enumerate}

\sbr
\noindent\underline{For $\alpha = \beta+1 < \kappa^+$}. For such $\alpha$, \emph{in addition} we choose $y_\alpha = y_{\beta+1} \in X_\ma$:

\begin{enumerate}[(i)]
\item \emph{to be new}:  ~~$y_{\beta+1} \in X_\ma \setminus \{ y_{\gamma+1} : \gamma < \beta \}$.

\item \emph{to be newly in the domain\footnote{Of course, it could have been in the domain earlier, we just need it to 
have fallen out for at least the previous step so that we are now free to set the distance estimate function as we wish.}}: ~~ $y_{\beta+1} \in c_\alpha(n_\alpha, 3)$ but $y_{\beta+1} \notin c_\beta(n_\beta, 3)$.

\item \emph{so that the domain stays small}: ~~ $|c_\alpha(n_\alpha, 3)| \leq^\cs |X_\ma \setminus c_\alpha(n_\alpha, 3)|$. 

\item \emph{so that the new distances are appropriate}:  ~~ for all $\gamma+1 < \beta+1$ and all $n$ such that 
$n_{\gamma+1} \leq_\mb n \leq_\mb n_\alpha$, we ask that  (recalling the bookkeeping function $g$)
\[ c_\alpha(n,4) (y_{\gamma+1}, y_{\beta+1}) = d_{g(\gamma+1, \beta+1)}.  \]
\end{enumerate}

\noindent We omit here the proof of the inductive construction, which the reader can find in the proof of \cite{MiSh:998} Theorem 8.1 (with the 
exception of condensing the numbering of the inductive hypotheses, we have kept the 
same notation as that proof so it should be possible to read it directly).  The star ingredient is local saturation, which 
reduces existence to finite consistency.   

Suppose then that we have carried out our inductive construction and have chosen the $c_\alpha$'s, $n_\alpha$'s, and $y_{\beta+1}$'s 
for $\alpha < \lambda$ and $\beta < \kappa^+$.  Let us finish the proof. As $\lambda = \xp_\cs < \xt_\cs$, we may choose 
a treetop $c_\star \in \mct$ above the sequence $\langle c_\alpha : \alpha < \lambda \rangle$. Remember that $c_\star$ is a 
function from $X_\mb$ to $X_\mb$, so $\langle n_\alpha : \alpha < \lambda \rangle$ is a strictly increasing sequence in 
$X_\mb$ below $n_\star$. By uniqueness, the co-initiality of the sequence $\langle n_\alpha : \alpha < \lambda \rangle$ 
in the set $\{ n : n \leq_\mb n_\star \}$ is $\kappa$, so we may find a cut 
\[ (\langle n_\alpha : \alpha < \lambda \rangle, \langle m_\epsilon : \epsilon < \kappa \rangle) \mbox{ in $X_\mb$}. \] 
Recall our original cut $(\bar{d}, \bar{e})$. 
As $c_\star(n,0)$ is strictly decreasing in $X_\ma$ as $n$ increases, each $c_\star(m_\epsilon, 0)$ is $\leq_\ma$ some $d_\gamma$. 
Without loss of generality, we may choose an increasing function $\zeta : \kappa \rightarrow \kappa$ such that 
$d_{\zeta(\epsilon)} <_\ma c_\star(m_\epsilon, 0) <_\ma d_{\zeta(\epsilon+1)}$. Now let's see where the constants have landed. 
For each $\beta < \kappa^+$, let 
\[ X_\beta = \{ n : n \leq_\mb n_* \mbox{ and }(\forall n^\prime)(n_{\beta+1} \leq_\mb n^\prime \leq_\mb n \implies 
y_{\beta+1} \in c_\star(n^\prime, 3)) \}\] 
record how long $c_{\beta+1}$ stayed continuously in the domain. By construction, 
$X_\beta$ includes the interval $[n_\alpha, n_{\alpha^\prime}]_\mb$ for all $\beta < \kappa^+$ and all 
$\beta < \alpha < \alpha^\prime < \lambda$. So each $X_\beta$ has a maximal element which is above all $n_\alpha$, and 
so because $(\bar{n}, \bar{m})$ is a cut, for some $\epsilon(\beta) < \kappa$, 
\[  [ n_{\beta+1}, m_{\epsilon(\beta)}]_\mb \subseteq X_\mb. \]
Since there are $\kappa^+$-many $\beta$'s, there are $W\subseteq \kappa^+$ of size $\kappa^+$ 
and $\epsilon_\star < \kappa$ such that $\epsilon(\beta) = \epsilon_*$ for all $\beta$ in $W$. 
Let $F = c_\star(m_{\epsilon_*}, 4)$ be the distance estimate function there. By construction, for every $\beta \neq \gamma \in W$, 
$F(y_{\gamma}, y_\beta) = d_{g(\gamma, \beta)}$. By the choice of $g$, there are $\gamma, \beta \in W$ such that 
\[ |d_{\zeta(\epsilon_*)+1} | <^\cs |d_{g(\gamma, \beta)}|. \]
This contradiction completes the proof. 
\end{proof}

So we arrive to: 

\begin{theorem}[\cite{MiSh:998} Theorem 9.1] \label{t:main} 
Let $\cs$ be a cofinality spectrum problem. Then $\mcs(\cs, \xt_\cs) = \emptyset$. 
\end{theorem}

\begin{proof}
There are two cases. 
If $\xp_\cs < \xt_\cs$, suppose $\kappa, \lambda$ are such that $\kappa + \lambda = \xp_\cs$ and 
$(\kappa, \lambda) \in \mcs(\cs, \xt_\cs)$. We have seen that without loss of generality, 
$\kappa \leq \lambda$, and that neither the case $\kappa = \lambda$ nor $\kappa < \lambda$ can occur. 
So $\xp_\cs < \xt_\cs$ cannot occur.  
So $\xt_\cs \leq \xp_\cs$, and we are done. 
\end{proof}

\begin{cor}[\cite{MiSh:998} Theorem 14.1] $\xp = \xt$.
\end{cor}

\begin{cor}[\cite{MiSh:998} Theorem 11.11] Any theory with $SOP_2$ is maximal in Keisler's order. 
\end{cor}

The paper \cite{MiSh:998} contains much that we haven't covered here, and several further consequences of Theorem \ref{t:main},  
including a new characterization of good ultrafilters.  

The reader may wonder: is this a one-time interaction of model theory and set theory or a beginning? In the sixties 
there was much interaction, but less later. These are 
exciting questions. The reader may wish to look at the recent paper of open problems \cite{MiSh:1069}. 
These arise largely from the 
methods and proofs described above,  rather than just the definitions and theorems, 
and it seems there is much more to be said.

\subsection*{Acknowledgments}
These notes were written as part of the NSF-funded Appalachian Set Theory workshop (DMS-1439507).  
M.M. was partially supported by NSF CAREER award 1553653. 
Thanks to S. Shelah for some very helpful comments on the notes.

\newpage 


\end{document}